%% file: computing_periods.tex
\documentclass[a4paper]{amsart}
\input{preamble.tex}

\author[E.~C.~Sert\"oz]{Emre Can Sert\"oz}
\address{Emre Can Sert\"oz\\
Max Planck Institute for Mathematics in the Sciences, Inselstr. 22\\ 04103 Leipzig, Germany}
\email{emresertoz@gmail.com}
\title{Computing periods of hypersurfaces}
\date{\today}

\begin{document}

\begin{abstract}
  We give an algorithm to compute the periods of smooth projective hypersurfaces of any dimension. This is an improvement over existing algorithms which could only compute the periods of plane curves. Our algorithm reduces the evaluation of period integrals to an initial value problem for ordinary differential equations of Picard--Fuchs type. In this way, the periods can be computed to extreme-precision in order to study their arithmetic properties. The initial conditions are obtained by an exact determination of the cohomology pairing on Fermat hypersurfaces with respect to a natural basis. 
\end{abstract}

\maketitle

\section{Introduction}

Let $X$ be a smooth complex projective hypersurface defined as the zero locus of a homogeneous polynomial with complex coefficients (see Remark~\ref{rem:effective_subfield} for more on the nature of these coefficients). The periods of $X$ can be concretely defined as integrals of rational functions on the ambient projective space as in Section~\ref{sec:intro_periods}. From a geometric point of view, the interest in periods arises because they determine the Hodge structure on the cohomology of $X$ and, when taken together with the intersection product on cohomology, periods often determine $X$ up to isomorphism. 
Indeed, there are properties which are much easier to discern from the periods of $X$ than from the defining equation of $X$. The computation of periods has been forbidding and there appears to have been no general purpose algorithm to compute periods when the dimension of $X$ is greater than one. In this paper, we give an algorithm which computes the periods of smooth hypersurfaces of any degree and dimension\footnote{An implementation of the algorithm is available here \url{https://github.com/emresertoz/PeriodSuite}.}. 

One striking area of application for the computation of periods is the determination of Picard rank of a surface in $\ppp$. This is a difficult problem which has received a great deal of attention~\cite{shioda-81,vL07, charles-14, schutt-15}. Nevertheless, there is no known algorithm for surfaces of degree greater than four and no implementations of the algorithm in~\cite{charles-14} for surfaces of degree four. On the other hand, if the periods and the intersection product of a surface are known, then its Picard rank is easily computed: see~\cite[Fact 3.11]{elsenhans-18} for a demonstration on K3 surfaces. We should note, however, that a \emph{reliable} computation of the Picard rank from the periods requires that the periods be computable to several hundreds of digits of precision. We give a detailed study of this application in another paper~\cite{lairez-sertoz}.

In higher dimensional algebraic geometry, periods appear in more subtle ways. Let us describe an instance where our algorithm performs especially well (see Section~\ref{sec:threefold}): degree three hypersurfaces in $\Pp^4$, that is, cubic threefolds. A celebrated result of Clemens and Griffiths~\cite{clemens-72} implies that no smooth cubic threefold is rational. Their proof builds upon the study of what they dubbed the intermediate Jacobian, which, in this case, is the quotient of $\Cc^5$ by the period vectors of the threefold. They also prove that the cubic threefold is determined up to isomorphism by its intermediate Jacobian. In light of this result, we expect the explicit computation of the periods of a cubic threefold, and hence its intermediate Jacobian, would be a fruitful source of experimentation.

Although it is not the emphasis of our algorithm, periods of curves deserve special mention---not least because of the role they played in the development of the modern understanding of periods. The Abel--Jacobi map, assigning to each curve its periods, continues to play an important role in algebraic geometry and its applications. 

For instance Guardia~\cite{guardia--explicit_genus_3} computed the periods of a particular family of genus three curves and studied their geometry, that is, the positions of the bitangents and the automorphism groups of these curves, using the periods. Extending Guardia's methods, Grushevsky and M\"oller~\cite{grush-moell--genus4} explicitly computed the periods of families of genus four curves in order to exhibit infinitely many Shimura varieties in the Jacobi locus.

There is also a remarkable connection between the periods of curves and integrable soliton equations~\cite{shiota--jacobian-soliton, holden-03}. In light of these connections, numerical and symbolic algorithms have been developed to study curves~\cite{tretkoff-84,frauendiener-17}. A crucial step involved in solving the soliton equations is the computation of periods of curves, for which various algorithms exist.

The first implementation of a general purpose algorithm to compute the periods of a curve was given by Sepp\"al\"a~\cite{seppala94} which was initially restricted to \emph{real} curves. Later, this restriction was lifted in~\cite{gianni--period_matrices}. In both of these algorithms, a curve is to be given as the quotient of a fundamental domain with respect to a Fuchsian group. This representation introduces a significant hurdle in determining the holomorphic forms of a curve. 

One can avoid this hurdle by starting with a plane model, as was done by the Maple package {\tt algcurves}~\cite{Maple10, deconinck01}. In the last few years, two more implementations of variations of this algorithm appeared in SageMath~\cite{sagemath} with stronger features. The first one is {\tt abelfunctions}~\cite{christopher16} and it can quickly evaluate theta functions associated to periods. The more recent implementation {\tt RiemannSurface}~\cite{nils-18} is now a built-in function of SageMath. The latter package comes with certified path tracking for the construction of homology on the curve and was motivated by number theoretic considerations, such as the computation of the endomorphism rings of Jacobians.

Availability of these algorithms allowed for the study of curves and their geometry in the spirit of \emph{numerical algebraic geometry}~\cite{hauenstein-17}. For example,~\cite{christopher16} computes the bitangents of a plane quartic from the periods using Riemann's formula~\cite{riemann-g3}. In a similar vein, one can recover curves in genus three and four from their periods using the formulas given in~\cite{guardia-11} and~\cite{kempf-86} respectively. This has been successfully implemented in genus four in~\cite{bernd-schottky}, providing a numerical solution to the Schottky problem~\cite{grush-12}.

The present work has been initiated by a problem posed by Bernd Sturmfels. He lists 20 problems in his Einstein Visiting Fellow proposal, the 20th of which asks for the \emph{development and implementation of a numerical method whose input is the equation of a quartic $Q$ in $\ppp$, and whose output is the corresponding point in the period domain $\Pp^{21}$}. 
  
We solve this problem, and we offer a generalization where the input can be the equation of a smooth projective hypersurface of any degree and dimension, with the period domain suitably expanded to capture all cohomological data. The expansion of the period domain is particularly useful in the study of hypersurfaces with no global holomorphic forms, such as the cubic threefold. 

In the case of curves, we outperform other software when very high precision is requested. This is because we numerically solve ordinary differential equations whereas the others numerically compute Riemann integrals, and it is much easier to attain high precision in the former case. We give an example of performance against precision in Figure~\ref{fig:time-precision} and in Figure~\ref{fig:extreme_precision}. In addition, we can compute the periods of higher genus curves if the defining equations of the curves have a particular shape. For instance, we present a computation of the periods of a curve of genus $1711$ at the end of Section~\ref{sec:higher_genus}.

\subsection{Periods as rational integrals}\label{sec:intro_periods}

Let $X \subset \Pp^{n+1}$ be an $n$-dimensional smooth hypersurface given as the zero locus of an irreducible homogenous polynomial $f_X$. The periods of $X$ are the values of certain integrals and so they are formed by two ingredients: the form to be integrated and the domain of integration. Let us introduce each of these in turn.

The forms to be integrated are based on the holomorphic volume form on projective space, which extends the volume form on its affine charts. On an affine space $\Cc^{n+1}$ with coordinates $z_1,\dots,z_{n+1}$ the standard holomorphic volume form is $\dd z_1  \dots  \dd z_{n+1}$. On the projective space $\Pp^{n+1}$ with homogeneous coordinates $x_0,\dots,x_{n+1}$ one defines the standard volume form to be
\begin{equation}\label{eq:Omega}
\Omega := \sum_{i=0}^{n+1} (-1)^i x_i \, \dd x_0  \dots  \widehat{\dd x_i}  \dots  \dd x_{n+1},
\end{equation}
where hat denotes the omission of that factor. Restricting to the standard affine charts, we see that the form $\Omega$ is locally the standard volume form (up to sign):
\[
\Omega|_{x_i=1} = (-1)^i \dd x_0 \dots \widehat{\dd x_i} \dots \dd x_{n+1}.
\]

\begin{notation}
  For any real $n$-dimensional cycle $\gamma \subset X$ we may construct a thin tube $\tau(\gamma)$ around $\gamma$ lying entirely in the complement of $X$ (see~\cite[\S 3]{griffiths--periods}).
\end{notation}

\begin{definition}
For any integer $\ell \ge 1$ and any polynomial $p(x) \in \Cc[x_0,\dots,x_{n+1}]$ of degree $d\ell - n -2$ the following integral is called a \emph{period of $X$}: 
\[
\frac{1}{2\pi \sqrt{-1}} \int_{\tau(\gamma)} \frac{p(x)}{f_X^\ell}\Omega,
\]
where $\gamma \subset X$ is a real $n$-cycle. 
\end{definition}

Notice that the integrand $\frac{p(x)}{f_X^\ell}\Omega$ is homogeneous of degree 0 with poles over $X$. Since the tube $\tau(\gamma)$ lies in the complement of $X$, the integrals are well defined and independent of the exact construction of the tube $\tau(\gamma)$.

\begin{remark}
  Informally, we can view $\rho:\tau(\gamma) \to \gamma$ as a circle bundle and then see this integral as computing the residue of the form at each point $p \in \gamma$ by taking a contour integral around the circle $\rho\inv(p)$. See Section~\ref{sec:compute_residues} where this method is employed. 
\end{remark} 

Taken in isolation, a single period does not reveal much of the structure of $X$. Instead, we need the integrals over all $n$-cycles. As the homology groups of a smooth hypersurface are freely generated, we now choose $n$-cycles $\gamma_1,\dots,\gamma_m \subset X$ whose classes generate the middle homology $\H_n(X,\Zz)$ of $X$. 

\begin{remark}
  If $X$ is a curve then $m$ here is simply twice the genus of $X$. If $X$ is a degree four surface then $m=22$ and if $X$ is a cubic threefold then $m=10$. We recall the general formula in Remark~\ref{rem:middle_dimension}.
\end{remark}

\begin{definition}
  For any integer $\ell \ge 1$ and any homogeneous polynomial $p(x) \in \Cc[x_0,\dots,x_{n+1}]$ of degree $d\ell - n -2$, the following complex valued vector is called a \emph{period vector} of $X$:
  \[
    \frac{1}{2\pi \sqrt{-1}} \left(  \int_{\tau(\gamma_1)} \frac{p(x)}{f_X^\ell}\Omega, \dots, \int_{\tau(\gamma_m)} \frac{p(x)}{f_X^\ell}\Omega \right) \in \Cc^m,
  \]
  where $H_n(X,\Zz)$ is freely generated by the $n$-cycles $\{\gamma_i\}_{i=1}^m$. 
\end{definition}

Most of these integrals are redundant, since any two forms that differ by an exact form will give the same period vector. In fact, one can construct a finite set of forms:
\begin{equation}
 \omega_i:= \frac{p_i(x)}{f_X^{\ell_i}}\Omega,\quad i=1,\dots,N,
  \label{eq:GD_basis}
\end{equation}
such that every other form $\frac{p(x)}{f_X^\ell}\Omega$ is a linear combination of $\omega_i$ modulo exact forms. In symbols, there are $a_1,\dots,a_N \in \Cc$ for which there is an equivalence modulo exact forms
\[
  \frac{p(x)}{f_X^\ell}\Omega \equiv a_1 \frac{p_1(x)}{f_X^{\ell_1}}\Omega + \dots + a_N \frac{p_N(x)}{f_X^{\ell_N}}\Omega.
\]
The procedure by which one such basis $\{\omega_i\}_{i=1}^N$ can be obtained, and the coefficients $a_i$ determined, is called the \emph{Griffiths--Dwork reduction}~\cite{griffiths--periods,dwork-62,lairez--periods}. 

\begin{remark}
  In fact, $N=m$ when $n=\dim_\Cc X$ is odd and $N=m-1$ when $n$ is even. This is because we are, in fact, computing a basis for the \emph{primitive} cohomology which coincides with cohomology when $n$ is odd or has one fewer generator when $n$ is even (see Section~\ref{sec:primitive}). 
\end{remark}

\begin{definition}\label{def:period_matrix}
  Let $\{\gamma_i\}_{i=1}^m$ be a basis for the middle homology $\H_n(X,\Zz)$ and $\{\omega_i\}_{i=1}^N$ a basis for the rational forms with poles along $X$, modulo exact forms, as in (\ref{eq:GD_basis}). Then the following $m\times N$ matrix is called a \emph{period matrix of $X$}:  
  \[
    \frac{1}{2\pi\sqrt{-1}} \left( \int_{\tau(\gamma_j)} \frac{p_i(x)}{f_X^{\ell_i}}\Omega \right)_{\substack{i=1,\dots,N \\ j=1,\dots,m}}.
  \]
\end{definition}

\begin{remark}
  A remarkable observation of Griffiths~\cite{griffiths--periods} is that the pole order $\ell$ and Hodge filtration on cohomology are closely related (see Theorem~\ref{thm:filtration} for a precise statement). This is what makes the period matrix of interest.
\end{remark}

Let us point out that the period matrix is of significantly greater value if the intersection matrix $(\gamma_i\cdot \gamma_j)_{i,j=1}^m$ is also known. We are now ready to state the precise nature of the problem that we solve. 

\begin{namedNum}{Goal}\label{goal}
  Given the defining polynomial $f_X$ of a smooth hypersurface $X \subset \Pp^{n+1}$, compute a period matrix as defined in Definition~\ref{def:period_matrix} and the intersection matrix $(\gamma_i\cdot \gamma_j)_{i,j=1}^m$.
\end{namedNum}

\subsection{Computing periods directly from the definitions}

The existing algorithms which compute the periods of plane curves attack this problem directly by constructing cycles generating the homology on a given curve and taking the relevant integrals numerically~\cite{deconinck01,nils-18}. The 1-cycles generating the homology of a plane curve $X \subset \pp$ are obtained by projecting $X$ to a coordinate axes $\p$, which realizes $X$ as a finite cover of $\p$, so that standard tools from topological covering theory may be used to compute the fundamental group $\pi_1(X)$ and therefore the middle homology $\H_1(X,\Zz)$~\cite{tretkoff-84}.

A similar technique could in principle be attempted in higher dimensions. For instance in~\cite{elsenhans-18}, a special case of surfaces is considered: K3 surfaces of rank 16 equipped with a double cover of the plane. There, explicit 2-cycles from the plane are lifted to the K3 surface. The intersection product is then determined numerically by repeated sampling. 

Replicating this approach in higher dimensions by giving explicit representations of $n$-cycles generating the homology will likely be very difficult. We solve the problem by giving the cycles \emph{implicitly}, carrying them over from simpler hypersurfaces in theory, and computing the values of the period integrals by solving ordinary differential equations of Picard--Fuchs type.

\subsection{Computing periods by homotopy} \label{sec:intro_ivp}

Let us suppose that we can compute all the periods of another smooth hypersurface $Y=Z(f_Y) \subset \Pp^{n+1}$ of degree $d$ with respect to a basis $\gamma_1^Y,\dots,\gamma_m^Y$ for the middle homology $\H_n(Y,\Zz)$ of $Y$. We will compute the periods of $X=Z(f_X)$ from those of $Y$ by what one might call a \emph{period homotopy}. That is, we will compute the periods of $X$ from those of $Y$ by deforming $Y$ into $X$ and tracing the value of the periods by ordinary differential equations. We will outline this construction in the remainder of the introduction and delegate some of the details to Section~\ref{sec:period_homotopy}.

Construct a family of hypersurfaces $\cx_t = Z(f_t) \subset \Pp^{n+1}$ varying algebraically with respect to a single complex parameter $t \in \Cc$ such that $\cx_0 = Y$ and $\cx_1 = X$. For instance, one could take $f_t = (1-t)f_Y + tf_X$ to be the defining equation of $\cx_t$. Except for finitely many values of $t$ in $\Cc$, the hypersurface $\cx_t$ is smooth. 

  Using Ehresmann's fibration theorem we may identify the topological spaces underlying the hypersurfaces $\cx_t$ along a path in the complex plane. We will choose this path so that it starts from $t=0$ and ends at $t=1$ while avoiding singular values of $t$. In particular, such an identification gives an isomorphism between the homology groups $\H_n(Y,\Zz)$ and $\H_n(\cx_{t},\Zz)$ for each $t$ along this path. This isomorphism is \emph{locally} canonical, not depending on the choice of the path or of the identification of the topological spaces. 

\begin{notation}
  By slight abuse of notation, let us omit mention of the path and denote by $\gamma_i(t) \in \H_n(\cx_t,\Zz)$ the homology class corresponding to $\gamma_i^Y$ via this isomorphism. In particular, the cycle classes $\gamma_i^X := \gamma_i(1)$ for $i=1,\dots,m$ form a basis for the homology of $X$.
\end{notation}

\begin{remark}
  The intersection product $(\gamma_i(t)\cdot \gamma_j(t))_{i,j}$ stays constant during deformation, so that we only need compute this intersection product on $Y$.
\end{remark}

\begin{remark}
  The approach here of constructing $\gamma_i(t)$ is conceptually similar to the tracking of zeros of polynomial systems commonly employed in numerical algebraic geometry. An important distinction is that we do not actually carry out this tracking procedure and use only the \emph{existence} of such $\gamma_i(t)$. Nevertheless, if effective representations for the cycles $\gamma_i^Y$ are given, as we do in Section~\ref{sec:pham_cycles}, one could sample points on these cycles and carry these points via homotopy continuation methods (e.g.\ using {\tt Bertini}~\cite{hauenstein-13} or {\tt HomotopyContinuation.jl}~\cite{homotopyjl}) in order to get a point cloud representation of $\gamma_i(t)$.
\end{remark}

\begin{remark}
  The basis $\gamma_i^X$ depends on the path chosen, however, only up to the \emph{discrete} action of the monodromy group on homology. In particular, for the rest of the discussion, the choice of the path can be ignored as we will deal with the infinitesimal variation of periods.
\end{remark}

 Let $p(x) \in \Cc[x_0,\dots,x_{n+1}]$ be of degree $d\ell - n -2$ for $\ell \ge 1$. We will now describe how to compute the period vector:
\[
  \frac{1}{2\pi \sqrt{-1}} \left(  \int_{\tau(\gamma^X_1)} \frac{p(x)}{f_X^\ell}\Omega, \dots, \int_{\tau(\gamma^X_m)} \frac{p(x)}{f_X^\ell}\Omega \right).
\]
Define the following rational form varying with $t$ and having poles along $\cx_t$:
\begin{equation}\label{eq:omegat}
  \omega(t) := \frac{p(x)}{f_t^\ell} \Omega.
\end{equation}
We can now define the following \emph{period function}:
\begin{equation}\label{eq:period_function_intro}
  \sigma(t) := \frac{1}{2\pi \sqrt{-1}} \left(  \int_{\tau(\gamma_1(t))} \frac{p(x)}{f_t^\ell}\Omega, \dots, \int_{\tau(\gamma_m(t))} \frac{p(x)}{f_t^\ell}\Omega \right).
\end{equation}
We will compute the value of interest, $\sigma(1)$, by finding an ODE satisfied by $\sigma(t)$ and then determining the initial conditions $\sigma(0),\sigma^{(1)}(0),\sigma^{(2)}(0),\dots$. From these, $\sigma(1)$ is found by numerical integration. 

\begin{remark}
  The strategy outlined here is remarkably similar to the deformation method used in the computation of zeta functions of hypersurfaces defined over finite fields~\cite{pancratz-tuitman}, where one computes the matrix of the Frobenius endomorphism by deforming it from the Fermat hypersurface using $p$-adic analysis.
\end{remark}

\subsection{Finding the differential equations} \label{sec:intro_de}
Classically, a differential equation satisfied by a period function is called a \emph{Picard--Fuchs equation}. In more recent incarnations, it is also referred to as a \emph{Gauss--Manin connection}. The Picard--Fuchs equations are central to geometry as well as mathematical physics~\cite{cox-katz--mirror}. For a perspective from the side of arithmetic geometry, let us refer the interested reader to the landmark paper~\cite{katz-70}. From a computational viewpoint, finding differential equations satisfied by integrals of functions is a heavily studied subject~\cite{chyzak-00,koutschan-10}. In particular, Lairez~\cite{lairez--periods} has dedicated software which can compute the Picard--Fuchs equations of rational integrals. 

We will now sketch how the Picard--Fuchs equation for $\sigma(t)$ is found. Denote the differentiation by $t$ operator with $\del_t$. The standard observation here is that we can differentiate under the integral sign:
\[
  \del_t^k \int_{\tau(\gamma_i(t))} \frac{p(x)}{f_t^\ell}\Omega = \int_{\tau(\gamma_i(t))} \del_t^k\frac{p(x)}{f_t^\ell}\Omega.
\]
Indeed, if $t'$ is very close to $t$ then $\gamma_i(t)$ and $\gamma_i(t')$ can be represented by cycles that are very close to one another. Then the tube $\tau(\gamma_i(t))$ encloses $\gamma_i(t')$ and is homotopic to the tube $\tau(\gamma_i(t'))$. Therefore, the domain of integration can be fixed while differentiating $\sigma(t)$.

In particular, suppose $\cd \in \Cc(t)[\del_t]$ is a differential operator such that
\begin{equation}\label{eq:D}
  \cd\omega(t) \equiv 0,
\end{equation}
where $\omega(t)$ is defined as in Equation (\ref{eq:omegat}) and equivalence is taken modulo exact forms. Then, the operator $\cd$ annihilates the period function $\sigma(t)$. 

The construction of an operator $\cd$ satisfying (\ref{eq:D}) is equivalent to finding a $\Cc(t)$-linear relation between the Griffiths--Dwork reductions of the derivatives $\omega^{(k)}(t)$, where the reduction is performed over the field $\Cc(t)$. Since these reduced forms must lie in a finite dimensional vector space, after finitely many differentiations and reductions we will find a $\Cc(t)$-linearly dependent set. The coefficients of the linear relation give $\cd$.  We present additional details on the construction of the Picard--Fuchs equations in Section~\ref{sec:PF}.

\begin{remark}\label{rem:effective_subfield}
  In truth, these operations are algorithmic if, instead of $\Cc$, we work over an \emph{effective} subfield of $\Cc$, where the equality of two elements can be checked in finite time. However, this is to be understood and will not be explicitly mentioned again. Our implementation works over the rational numbers but there are no significant hurdles to extending it to number fields.
\end{remark}

\subsection{Computing the initial values}\label{sec:intro_iv}

It remains to find one smooth hypersurface $Y$ for each $n, d \ge 1$ for which we can compute the periods. Fixing $n$ and $d$, let us denote by $Y$ the Fermat hypersurface cut out by the equation $f_Y := x_0^d + \dots +x_n^d - x_{n+1}^d$.

The Fermat hypersurface has symmetries which we will exploit. In particular, the following automorphisms will be needed.

\begin{definition}
  Let $\xi := e^{\frac{2\pi \sqrt{-1}}{d}}$ be a $d$-th root of unity. For each $\beta = (\beta_0,\dots,\beta_{n+1}) \in \Zz^{n+2}$ we define a \emph{translation} $t^\beta: Y \to Y$ on $Y$ defined by scaling each of the coordinates by powers of $\xi$ as follows:
  \[
    t^\beta:  [x_0, \dots , x_{n+1}] \mapsto [\xi^{\beta_0}x_0,\dots,\xi^{\beta_{n+1}}x_{n+1}].
  \]
\end{definition}

\begin{notation}
  In $Y$ there is an $n$-cycle $S \subset Y$, homeomorphic to a sphere, called the \emph{Pham cycle}  whose translates generate the relevant part of the homology (see Section~\ref{sec:pham_cycles}). 
\end{notation}

When $n$ is odd, the set of translates $\{t^\beta S \mid \beta \in \Zz^{n+2}\}$ of $S$ generates the homology $\H_n(Y,\Zz)$. In fact, we can find a finite subset $B \subset \Zz^{n+2}$ such that the set $\{t^\beta S \mid \beta \in B\}$ freely generates the homology. The determination of this subset $B$ relies on Corollary~\ref{cor:primitive_basis}.

When $n$ is even, the Pham cycle and its translates do not generate the homology and we need one other cycle. Let $L$ be the image of the following linear map:
\begin{align} \label{eq:L}
  \Pp^{\frac{n}{2}} &\to \Pp^{n+1} \\
  [u_0,\dots,u_{\frac{n}{2}}] &\mapsto [u_0,\eta u_0, u_1, \eta u_1, \dots, u_{\frac{n}{2}},\eta u_{\frac{n}{2}}], \nonumber
\end{align}
where $\eta := e^{\frac{\pi \sqrt{-1}}{d}}$ is a $d$-th root of $-1$. The translates of $S$ together with $L$ generate the middle homology of $Y$, see Lemma~\ref{lem:generators}. Once again, using Corollary~\ref{cor:primitive_basis}, a finite subset $B \subset \Zz^{n+2}$ can be found algorithmically so that $\{t^\beta S \mid \beta \in B\} \cup \{ L\}$ freely generates $\H_n(Y,\Zz)$.

However, for period computations, this extra class $L$ is redundant. The periods over $t^\beta S$ will determine the periods over $L$ by simple linear algebra as we explain in Section~\ref{sec:primitive_to_entire}. In light of these statements, the following theorem allows for the computation of all periods on Fermat hypersurfaces.

\begin{notation}
  Let $\Gamma(x) = \int_0^\infty e^{-t} t^{x-1} \dd t$ denote the Euler gamma function.
\end{notation}

\begin{theorem*}[Theorem~\ref{thm:initials}]
  Let $Y \subset \Pp^{n+1}$ be the degree $d$ Fermat hypersurface given by the equation $f_Y = x_0^d + \dots + x_n^d- x_{n+1}^d$. Let $\alpha \in \Zz_{>0}^{n+2}$ where $|\alpha| = \ell d$ for some $\ell \in \Zz_{>0}$ and let $\beta \in \Zz^{n+2}$. Then we have:
  \[
    \frac{1}{2\pi\sqrt{-1}} \int_{\tau(t^\beta S)} x_0^{a_0-1}\cdots x_{n+1}^{a_{n+1}-1} \frac{\Omega}{f_Y^\ell} =   -\prod_{j=1}^{\ell-1}\left(1-\frac{a_{n+1}}{jd}\right) \prod_{i=0}^n\left( \frac{1-\xi^{-\alpha_i}}{d} \right)  \frac{\prod_{i=0}^n \Gamma(\frac{\alpha_i}{d})}{\Gamma(\sum_{i=0}^n \frac{\alpha_i}{d})} \xi^{\alpha\cdot \beta},
  \]
  where $\alpha \cdot \beta$ denotes the dot product $\sum_{i=0}^{n+1} \alpha_i \beta_i$.
\end{theorem*}

\begin{remark}
  The left hand side appears symmetric with respect to the exponents $a_i$, whereas the last exponent $a_{n+1}$ plays a very distinctive role on the right hand side. This is because the Pham cycle $S$ breaks the symmetry as its construction is performed on the affine chart $x_{n+1}=1$.
\end{remark}

\begin{remark}
  In the case of curves, $n=1$, this result was first obtained by Rohrlich in the appendix to Gross' paper~\cite{gross--abelian_integrals}. Later, Tretkoff~\cite{tretkoff--periods_of_fermat} computed these values for surfaces, $n=2$. It appears that the periods of the Fermat hypersurface for any dimension have been first computed by Deligne~\cite[\S I.7]{dmos-82}, with a thorough treatise to appear in~\cite[\S 15.2]{movasati--hodge}. Combining these results with the residue computations in~\cite{carlson-80} would then recover the equation above. Nevertheless, the formula relies on elementary analysis and we compute it in Section~\ref{sec:periods_of_fermat}.
\end{remark}

\begin{remark}
  The intersection products of the Pham cycles $t^\beta S$ are well known and attributed to Pham~\cite{pham--fermat}. For a modern treatment see any one of~\cite{arnold-1984, movasati--hodge, looijenga-10}. 
\end{remark}

Now that we have explicit formulas for the periods of the degree $d$ Fermat hypersurface $Y \subset \Pp^{n+1}$, the periods of any other degree $d$ hypersurface $X \subset \Pp^{n+1}$ can be expressed as initial value problems using the strategy outlined in Sections~\ref{sec:intro_ivp} and~\ref{sec:intro_de}. It remains to solve these initial value problems by numerical integration.

\subsection{Numerical integration}

The ODEs we obtain from the procedure outlined in Section~\ref{sec:intro_de} are typically huge. They require dedicated software to be integrated efficiently. We use Marc Mezzarobba's {\tt analytic} extension of the {\tt ore\_algebra} package in {\tt SageMath} which has the added benefit of providing rigorous error bounds for the result~\cite{mezzarobba-oa}. As far as we are aware, no other software can handle ODEs that would take many pages to write down. 

Currently, the determination of the ODEs and of the initial conditions is fully automated via our implementation in {\tt Magma}~\cite{magma}. A simple script allows for the output to be integrated by Mezzarobba's solver.

\subsection*{Acknowledgments}

Bernd Sturmfels suggested this problem and provided guidance, support and motivation. Pierre Lairez generously made many expository and technical contributions to this project, in particular, he introduced me to LLL and to Mezzarobba's work. I've also benefited from conversations with numerous other mathematicians. In particular, I would like to thank: Alex Degtyarev for introducing me to Pham cycles and answering related questions. Marc Mezzarobba for helpful comments on how to use his code. Mateusz Micha\l{}ek for helping with Gr\"obner bases and combinatorics. Kristin Shaw for finding the paper which got this project off the ground. Lynn Chua for carefully reading through the first draft. In addition, I would like to thank Javier Fresan, Hossein Movasati and Jan Tuitman for valuable comments on the first version of this paper. Finally, heartfelt thanks to the referee for a careful reading of the manuscript and numerous suggestions for improvement.

\section{Foundational material}\label{sec:basics}

In this section we will review the standard results that we rely on for a more detailed description of the period homotopy algorithm given in Section~\ref{sec:period_homotopy}. Let $X \subset \Pp^{n+1}$ be a smooth degree $d$ hypersurface, defined as the zero locus of an irreducible polynomial $f_X \in \Cc[x_0,\dots,x_{n+1}]$.

\subsection{Primitive (co)homology}\label{sec:primitive}

Understanding the homology and cohomology of $X$ is greatly simplified by passing to its complement $\Pp^{n+1}\setminus X$. However, by this passage we lose a small amount of information and describe the \emph{primitive} (co)homology instead.

\begin{notation}\label{not:h}
  If $n=\dim_{\Cc} X$ is even, let $h \in \H_n(X,\Zz)$ be the class of the intersection $X \cap \Lambda$ where $\Lambda \subset \Pp^{n+1}$ is a linear space of codimension $\frac{n}{2}$. If $n$ is odd set $h=0$. 
\end{notation}

Recall that we have a non-degenerate intersection pairing:
\begin{align*}
  \H_n(X,\Zz) \times \H_n(X,\Zz) &\to \Zz \\
  (\gamma_1,\gamma_2) &\mapsto \gamma_1\cdot\gamma_2,
\end{align*}
which is alternating if $n$ is odd and symmetric if $n$ is even. On cohomology we have the dual cup product.

\begin{definition}
    The subgroup of homology formed by cycle classes orthogonal to $h$ is called the \emph{primitive homology of $X$} and denoted:
    \[
      \PH_n(X,\Zz) := \{\gamma \in \H_n(X,\Zz) \mid \gamma \cdot h = 0\}.
    \]
    Similarly, the elements in cohomology $\H^n(X,\Zz)$ which annihilate $h$ form the \emph{primitive cohomology of $X$}, denoted $\PH^n(X,\Zz)$. Define $\PH_n(X,\Kk)$ and $\PH^n(X,\Kk)$ in a similar fashion for $\Kk$ a field.
\end{definition}

When $n$ is odd, primitive (co)homology coincides with (co)homology. When $n$ is even, $h\cdot h =d$ and thus $h \notin \PH_n(X,\Zz)$. However, we need to add more than just $h$ into the primitive homology in order to recover homology.

\begin{lemma}
  When $n$ is even, there is a class $v \in \PH_n(X,\Zz)$ such that $h\cdot v=1$.
\end{lemma}
\begin{proof}
  By Ehresmann's fibration theorem, any two hypersurfaces of degree $d$ and dimension $n$ are equivalent from a homological point of view. Therefore, we may assume that $X$ is the Fermat hypersurface. Then we may take $v$ to be the class of any linear space contained in $X$, such as the one given in Equation (\ref{eq:L}). Clearly, $v \cdot h =1$.
\end{proof}

\begin{lemma}\label{lem:generators}
When $n$ is even, the subgroup $\PH_n(X,\Zz) \oplus \Zz\langle h \rangle$ has index $d$ in $\H_n(X,\Zz)$. For $v \cdot h =1$ we have $\PH_n(X,\Zz) \oplus \Zz\langle v \rangle = \H_n(X,\Zz)$. 
\end{lemma}
\begin{proof}
  Since $v \cdot h = 1$ but $h\cdot h = d$ it is clear that $v \notin \PH_n(X,\Zz) \oplus \Zz\langle h \rangle$.  On the other hand, given any class $u \in \H_n(X,\Zz)$ we can define $u' := u - (u\cdot v) h \in \PH_n(X,\Zz)$ and write $u = u' + (u\cdot v)h \in \PH_n(X,\Zz)\oplus \Zz\langle v \rangle$. This implies that the inclusion $\PH_n(X,\Zz)\oplus \Zz\langle v \rangle \toi \H_n(X,\Zz)$ is an isomorphism. In particular, $\PH_n(X,\Zz)\oplus\Zz\langle h  \rangle$ has index $d = \deg X$ in $\H_n(X,\Zz)$.
\end{proof}

\begin{remark} \label{rem:v}
  In homology with rational coefficients, $\H_n(X,\Qq)$, we have $h \cdot (v - \frac{1}{d}h)=0$. Therefore, we can express $v$ as a sum $\frac{1}{d}h + \gamma$ where $\gamma \in \PH_n(X,\Qq)$.
\end{remark}

  Let $U := \Pp^{n+1} \setminus X$ be the complement of $X$. Recall that for any $n$-cycle $\gamma \subset X$ we defined by $\tau(\gamma) \subset U$ the $(n+1)$-cycle obtained by forming a thin tube, i.e., $S^1$-bundle, over $\gamma$. This induces the following map between homology groups:
  \begin{align}\label{eq:tube}
  \H_{n}(X,\Zz) &\to \H_{n+1}(U,\Zz) \\
  [\gamma] &\mapsto [\tau(\gamma)].\nonumber
\end{align}
In cohomology, the corresponding map is:
\begin{align}\label{eq:cotube}
  \H_{\dR}^{n+1}(U,\Cc) &\to \H^n(X,\Cc) \\
  \omega &\mapsto (\gamma \mapsto \frac{1}{2\pi\sqrt{-1}}\int_{\tau(\gamma)} \omega).\nonumber
\end{align}

\begin{proposition}
  The natural map $\H^{n+1}(U,\Cc) \to \H^n(X,\Cc)$ defined above establishes an isomorphism:
  \[
    \H^{n+1}(U,\Cc) \isoto \PH^n(X,\Cc).
  \]
\end{proposition}
\begin{proof}
  This follows from the excision sequence in topology. See page 159 of~\cite{voisin-2007-volII}.
\end{proof}

\subsection{Griffiths residues} \label{sec:griffiths_residues}

Let $V$ be the space of all holomorphic top forms on $U=\Pp^{n+1} \setminus X$, i.e., $V = \H^0(U,\Omega^{n+1}_{U/\Cc})$. Let us denote by $V_\ell \subset V$ the subspace of forms which extend to $\Pp^{n+1}$ with a pole order at most $\ell$ on $X$. In symbols, $V_\ell = \H^0(\Pp^{n+1},\co_{\Pp^{n+1}}(\ell[X]))$. Then we have $V = \bigoplus_{\ell \ge 1} V_\ell$.

\begin{remark}\label{rem:Vl}
  Let us point out that any element in $V_\ell$ can be written as a quotient:
  \[
   \frac{p}{f_X^\ell}\Omega,
  \]
  where $p \in \Cc[x_0,\dots,x_{n+1}]$ is of degree $d \ell  -n -2$, $\Omega$ is the volume form on $\Pp^{n+1}$ defined in Equation (\ref{eq:Omega}) and $f_X$ is the equation defining $X$.
\end{remark}

\begin{proposition}\label{prop:residue_map_is_surjective}
  The natural map $V \to \H^{n+1}(U,\Cc)$ is surjective. The kernel of this map is generated by exact forms in $V$.
\end{proposition}
\begin{proof}
  This is a particular case of Theorem 6.4~\cite{voisin-2007-volII}.
\end{proof}

\begin{definition}
  The composition of the maps $V \tos \H^{n+1}(U,\Cc) \isoto \PH^n(X,\Cc)$ is called the \emph{residue map}, and will be denoted by $\res: V \tos \PH^n(X,\Cc)$. The restriction of the residue map to $V_\ell$ will be denoted by $\res_\ell$. Equations (\ref{eq:tube}) and (\ref{eq:cotube}) imply the following identity:
  \begin{equation}\label{eq:the_identity}
    \int_{\gamma} \res \omega = \frac{1}{2\pi\sqrt{-1}}\int_{\tau(\gamma)} \omega,
  \end{equation}
  for any $\gamma \in \H_n(X,\Zz)$ and $\omega \in V$.
\end{definition}

\begin{definition}\label{def:residue_basis}
  A set of forms $\omega_1,\dots,\omega_N \in V$ will be said to form a \emph{residue basis for $X$} if their residues $\res \omega_1, \dots, \res \omega_N$ form a basis in $\PH^n(X,\Cc)$.
\end{definition}

The cohomology of $X$ admits the Hodge decomposition $\H^n(X,\Cc) = \bigoplus_{k=0}^n \H^{n-k,k}(X,\Cc)$. The corresponding Hodge filtration is denoted by:
\[
  F^\ell(\H^n(X,\Cc)) := \bigoplus_{k=0}^{\ell} \H^{n-k,k}(X,\Cc).
\]
This filtration on cohomology induces a filtration on the primitive cohomology by restriction:
\[
  F^\ell(\PH^n(X,\Cc)) := \PH^n(X,\Cc) \cap F^\ell(\H^n(X,\Cc)).
\]

\begin{theorem} \label{thm:filtration}
  The residue map $\res_\ell: V_\ell \to \PH^n(X,\Cc)$ surjects onto $F^{\ell-1}(\PH^n(X,\Cc))$.
\end{theorem}
\begin{proof}
  This is Theorem 8.1 of~\cite{griffiths--periods}.
\end{proof}

\begin{remark}
Notice that there are no exact forms in $V_1$ as the derivative of a rational form can not produce pole order one. Combined with the theorem above we get an \emph{isomorphism}:
  \[
    \res : V_1 \isoto \H^{n,0}(\cx_t, \Cc) \toi \PH^n(\cx_t,\Cc).
  \]
\end{remark}

\subsection{Kernel of the residue map}

There is a finite procedure to maximally reduce an element in $V$ modulo the kernel of the residue map. This reduction procedure is based on the following two theorems. The first theorem states that we do not need infinitely many pole orders to describe the primitive cohomology of $X$.

\begin{theorem}[Theorem 4.2~\cite{griffiths--periods}] \label{thm:bdd_pole_order}
  The restricted residue map $\res_{n+1} : V_{n+1} \to \PH^n(X,\Cc)$ is surjective where $n=\dim_\Cc X$.
\end{theorem}

The second theorem makes the kernel of the residue map more explicit and is a strengthening of Proposition~\ref{prop:residue_map_is_surjective} above. Let $W_\ell = \H^0(\Pp^{n+1},\Omega^n_{\Pp^{n}/\Cc}(\ell[X]))$ be the space of $n$-forms on $\Pp^{n+1}$ with pole order at most $\ell$ on $X$. Once again, these can be viewed as holomorphic $n$-forms on $U$. We have the algebraic derivation map $\dd: W \to V: \eta \mapsto \dd \eta$ which allows us to write:
\[
  \H^{n+1}(U,\Cc) \simeq V/\dd(W).
\]

\begin{theorem}[Theorem 4.3~\cite{griffiths--periods}]\label{thm:griffiths2}
  For each $\ell \ge 1$, the kernel of the restricted residue map $\res_\ell : V_\ell \to \PH^n(X,\Cc)$ is exactly the image of the derivation map $\dd : W_{\ell-1} \to V_\ell$. 
\end{theorem}
\begin{remark}
  It is clear that $\dd(W_{\ell-1})$ must belong to the kernel of $\res_\ell$. What is non-trivial here is that, if for some $k \ge \ell$ a form $\eta \in W_k$ has derivative $\dd \eta$ in $V_\ell$, then we can find $\tilde \eta \in V_{\ell-1}$ such that $\dd \eta = \dd \tilde \eta$. 
\end{remark}

\subsection{Griffiths--Dwork reduction}\label{sec:GD}

Given a form $\omega \in V$ we will describe the \emph{Griffiths--Dwork reduction} on $\omega$ which modifies $\omega$ modulo the kernel of the residue map and puts it into a \emph{reduced form}. The primary motivation for this reduction is Lemma~\ref{lem:linear_dependence}. 

Recall the nature of the elements in $V$ from Remark~\ref{rem:Vl}. We will now describe the elements in $W$, the holomorphic $n$-forms on $\Pp^{n+1} \setminus X$. Let us use $\dd x^{i,j}$ to refer to the $n$-form $\dd x_0 \dots \widehat{\dd x_i} \dots \widehat{\dd x_j} \dots \dd x_{n+1}$, where hat denotes omission. A form in $\varphi \in W_\ell$ can be written as:
\begin{align}\label{eq:dphi}
  \varphi &= \sum_{i < j}(-1)^{i+j}(x_i s_j(x) - x_js_i(x)) \frac{\dd x^{i,j}}{f_X^\ell},
\end{align}
where $s_i \in \Cc[x_0,\dots,x_{n+1}]$ is homogeneous of degree $(\ell d  - n -1)$ (see~\cite{griffiths--periods} for a derivation of this fact).

Let us say that a form $\omega \in V$ has pole order $\ell$ if $\omega \in V_{\ell}$ but $ \omega \notin V_{\ell -1}$. Pole order induces a natural grading on $V$. If $\omega \in V$ we will write $\omega = \sum_{i=1}^{r} \omega_j$ where $\omega_j$ is zero or has pole order $j$, and the highest term $\omega_r$ is non-zero. Let us refer to $r$ as the pole order of $\omega$. Similarly, we can grade $W$ by pole order.

\begin{notation}\label{not:jacobian}
  Let $J(f_X)=(\frac{\del f_X}{\del x_0}, \dots,\frac{\del f_X}{\del x_{n+1}})$ in $\Cc[x_0,\dots,x_{n+1}]$ be the Jacobian ideal of $f_X$. We will assume a Gr\"obner basis for $J(f_X)$ is fixed once and for all, so that all remainder computations modulo $J(f_X)$ are well defined.
\end{notation}

Given $\omega = \sum_{i=1}^r \omega_i \in V$, for each $i=1,\dots,r$ write $\omega_i = p_i(x) \frac{\Omega}{f_X^i}$ where $p_i$ is a polynomial of degree $id - n- 2$ not divisible by $f_X$. Let $q_r$ be the remainder of $p_r$ modulo $J(f_X)$ and find polynomials $s_0,\dots,s_{n+1}$ satisfying:
\[
  p_r - q_r = \sum_{j=0}^{n+1}s_j \frac{\del f_X}{\del x_j}.  
\]
Define the polynomial $q_{r-1} = \frac{1}{r-1} \sum_{j=0}^{n+1} \frac{\del s_j}{\del x_j}$. Then by Equation~\ref{eq:dphi} we may conclude that $\omega$ is equivalent modulo exact forms to
\[
  \tilde \omega := \frac{q_r \Omega}{f_X^r} + \frac{(p_{r-1}+q_{r-1})\Omega}{f_X^{r-1}} + \sum_{i=1}^{r-2} \omega_i.
\]
The choice of the ``coordinates'' $s_i$ for the difference $p_r-q_r$ is not canonical and therefore the partially reduced form $\tilde \omega$ is not canonical. However, an application of Theorem~\ref{thm:griffiths2} implies that any other choice of coordinates would yield a reduction that is equivalent to $\tilde \omega$ modulo exact forms of pole order $r-1$. This difference will vanish in the next step of the reduction.

Writing $\tilde \omega = \sum_{i=1}^r \tilde \omega_i$, and $\tilde \omega_i = \frac{\tilde p_i \Omega}{f_X^i}$, we can apply the same reduction method to the form $\tilde \omega_{r-1}$. After repeated application, all terms in the pole decomposition have been put into a normal form with regards to $J(f_X)$. 

\begin{definition}
  The resulting form is denoted by $[\omega]_{\GD}$ and will be called a \emph{reduced form}.  
\end{definition}

\begin{lemma}\label{lem:linear_dependence}
  Suppose $\vphantom{\omega}^{1}\omega,\dots,\vphantom{\omega}^{s}\omega$ are reduced forms in $V$. Then the $\vphantom{\omega}^{i}\omega$ are linearly independent over $\Cc$ if and only if the cohomology classes $\res \vphantom{\omega}^{i}\omega$ are linearly independent over $\Cc$.
\end{lemma}
\begin{proof} The ``if'' direction is trivial. We prove the other implication by proving its contrapositive. In fact, we will prove that any linear relation amongst the residues lifts to a linear relation on the reduced forms. Suppose $\sum_{i=1}^s a_i \res \vphantom{\omega}^{i}\omega = 0$ for some $a_i \in \Cc$ and define $\eta=\sum_{i=1}^s a_i \vphantom{\omega}^{i}\omega = \sum_{j=1}^r \eta_j$ where $r$ is the maximal pole order amongst all $\vphantom{\omega}^{i}\omega$. We will show by descending induction on pole order that $\eta_j =0$ for all $j$. Since $\res \eta = 0$, we can reduce the pole order of $\eta$ modulo the kernel of the residue map. This means that the highest order term $\eta_r = \frac{p_r \Omega}{f^r}$ must satisfy $p_r \in J(f_X)$. On the other hand, the highest order terms of $\vphantom{\omega}^{i}\omega$ are all in normal form with respect to $J(f_X)$ and thus a linear combination belongs to $J(f_X)$ if and only if the linear combination is 0, implying $p_r =0$. We may now proceed to pole order $r-1$, and since $\eta_r=0$, $\res \eta =0$ implies $\eta_{r-1} \in J(f_X)$. We may thus repeat the argument for all terms $\eta_j$.
\end{proof}

Let $R$ be the polynomial ring $\Cc[x_0,\dots,x_{n+1}]$. We will denote by $R_s \subset R$ the subspace of homogeneous degree $s$ polynomials in $R$. For a homogeneous ideal $I \subset R$, we denote by $I_s$ the intersection $I \cap R_s$. Define the following vector space:
\[
  S = \bigoplus_{\ell \ge 1} R_{d\ell-n-2}/J(f_X)_{d\ell-n-2}.
\]
Observe that $S \subset R/J(f_X)$ is finite dimensional since $Z(f_X)$ is smooth.

  Let $p_1,\dots,p_N \in \Cc[x_0,\dots,x_{n+1}]$ be homogeneous polynomials which descend to a basis on $S$. For each $i$, let $\ell_i$ be the positive integer satisfying $\deg p_i = d\ell_i -n -2$. Define the forms
\[
  \omega_i := \frac{p_i}{f^{\ell_i}}\Omega,\quad i=1,\dots,N.
\]

\begin{proposition}\label{prop:GD-basis_residues}
  The forms $\omega_1,\dots,\omega_N$ form a residue basis for $X$, see Definition~\ref{def:residue_basis}. 
\end{proposition}

\begin{proof}
  This follows from Lemma~\ref{lem:linear_dependence} above and the observation that none of the $\omega_i$ can be reduced any further. See also Proposition 6.2 in~\cite{voisin-2007-volII}.
\end{proof}

\begin{remark}
  For convenience, we explained the Griffiths--Dwork reduction over $\Cc$. However, the reduction algorithm works over any field of characteristic 0. We need the restriction on the characteristic because of the division required during the reduction process.
\end{remark}

\subsection{Picard--Fuchs equations}\label{sec:PF}

Fix an integer $n > 0$ and let $f \in \Cc[s][x_0,\dots,x_{n+1}]$ be a homogeneous polynomial of degree $d$ with coefficients which are polynomial in the variable $s$, further assume that the coefficients of $f$ are relatively prime. For $t \in \Cc$ we will write $f_t:=f|_{s=t}$, $\cx_t := Z(f_t) \subset \Pp^{n+1}$ and $U_t := \Pp^{n+1} \setminus \cx_t$. Let us assume that $\cx_t$ is smooth for generic $t$.

Fix a non-zero polynomial $p \in \Cc[x_0,\dots,x_{n+1}]$ of degree $d \ell - n -2$ where $\ell \ge 1$ is an integer. Define the following holomorphic form on $U_t$:
\[
  \omega(t) := \frac{p(x)}{f_t^\ell}\Omega.
\]
We will now describe how to find the Picard--Fuchs equation corresponding to $\omega(t)$.

Recall that we denote by $\del_t$ the differentiation by $t$ operator. Let $\omega^{(k)}(t) := \del^k_t \omega(t)$ be the $k$-th derivative of $\omega(t)$ with respect to $t$. For any $k=0,1,\dots$ the Griffiths--Dwork reduction $[\omega^{(k)}(t)]_{\GD}$ of the $k$-th derivative has pole order bounded by $n+1$ due to Theorem~\ref{thm:bdd_pole_order}. In particular, the reductions of these infinitely many derivatives live in a finite dimensional $\Cc(t)$-vector space. 

Let $\delta \ge 1$ be the smallest integer for which there is a linear linear relation:
\begin{equation}\label{eq:linear_relation}
  [\omega(t)^{(\delta)}]_{\GD} +a_{\delta-1}(t)[\omega(t)^{(\delta-1)}]_{\GD} +\dots+ a_0(t) [\omega(t)]_{\GD} = 0,
\end{equation}
where $a_i(t)\in \Cc(t)$ are rational functions in $t$.

\begin{notation}\label{not:pf}
  Let $\cd = \del_t^\delta + \sum_{i=0}^{\delta-1} a_i(t) \del_t^i \in \Cc(t)[\del_t]$ be the corresponding differential operator. 
\end{notation}

\begin{lemma}
  Modulo exact forms, we have the following equivalence:
  \[
    \cd \omega(t) \equiv 0.
  \]
\end{lemma}
\begin{proof}
  Indeed, in Equation (\ref{eq:linear_relation}) we may simply substitute $[\omega^{(k)}(t)]_{\GD} \equiv \omega^{(k)}(t)$ for each $k=0,\dots,\delta$.
\end{proof}

 \begin{definition}
   The ODE constructed above, $\cd=0$, is called the \emph{Picard--Fuchs equation} corresponding to the form $\omega$.
 \end{definition}

 From now on assume that $\cx_0$ is a smooth hypersurface. Furthermore, we will assume that a basis for the primitive homology is fixed around $t=0$, which we will denote by $\{\gamma_i(t)\}_{i=1}^m \subset \PH_n(\cx_t,\Zz)$. As in the introduction (\ref{eq:period_function_intro}), define the \emph{period function} associated to $\omega(t)$:
 \begin{equation}\label{eq:period_function}
   \sigma(t) := \left(\int_{\gamma_1} \res \omega(t), \dots, \int_{\gamma_m} \res \omega(t)\right),
 \end{equation}
 and the individual periods $\sigma_i(t) = \int_{\gamma_i}\res \omega(t)$.

 \begin{lemma}
   The operator $\cd$ defined above is the minimal annihilating differential operator for the period function $\sigma(t)$.
 \end{lemma}
 \begin{proof}
   Suppose $\cd'\in \Cc(t)[\del_t]$ is another monic operator annihilating $\sigma(t)$ with order no greater than the order of $\cd$. Then the form $\res (\cd'\omega)$ integrates to 0 over all $\gamma_i(t)$. But $\gamma_i(t)$ span the primitive homology, which is dual to the primitive cohomology to which $\res \omega$ belongs. This forces $\res (\cd' \omega)= 0$ and thus $[\cd'(\del_t)\omega]_{\GD}=0$. By definition of $\cd$ we must have $\cd=\cd'$.
 \end{proof}

 The coefficients of $\cd$ may have poles at $t=0$. However, the poles at $t=0$ are always mild, which for us means that the space of solutions of $\cd=0$ is generated by holomorphic functions around $t=0$. Before we start this proof, define the following space of functions generated by periods associated to $\omega$:
\[
 \cl=\Cc\langle \sigma_{i}(t) \mid i=1,\dots,m \rangle.
\]

\begin{proposition}\label{prop:sols_are_holo}
   The solution space of $\cd=0$ near $t=0$ is $\cl$. In particular, all solutions of this ODE are holomorphic at $t=0$.
 \end{proposition}

 \begin{proof}
   By construction, $\cd$ annihilates the vector $\sigma(t)=(\sigma_1(t),\dots,\sigma_m(t))$. Therefore, $\cl$ is contained in the solution space of $\cd=0$. This implies $\dim_{\Cc} \cl\le \delta$.

   On the other hand, minimality of $\cd$ implies that the derivatives $\sigma(t),\sigma(t)^{(1)},\dots,\sigma(t)^{(\delta-1)}$ are \emph{not} linearly dependent over $\Cc(t)$. Any linear relation satisfied by the entries of $\sigma(t)$ would continue to be satisfied by the entries of the derivatives of $\sigma(t)$. In particular, all the derivatives $\sigma^{(k)}(t)$ must lie in a fixed sub-space of $\Cc^{m}$ having dimension $\dim \cl$. This forces $\dim_\Cc \cl\ge \delta$ completing the proof.
 \end{proof}

 \begin{remark}
   Numerous properties of Picard--Fuchs equations are explored in depth in the book~\cite{arnold-1984}. The previous proof is adapted from that of Theorem 12.2.1 in \emph{loc.\ cit.}
 \end{remark}

 \subsection{Computing the initial conditions}\label{sec:initial_conditions}

Since the Picard--Fuchs equation $\cd=0$ of degree $\delta$ defined in Notation~\ref{not:pf} may have a singularity at $t=0$, the first $\delta$ derivatives of $\sigma(t)$ at $t=0$ may not be sufficient to start integration. Nevertheless, Proposition~\ref{prop:sols_are_holo} implies that we can find a basis for the solution space whose power series expansions near $t=0$ will have the following leading terms:
\[
  t^{u_1}, t^{u_2},\dots,t^{u_{\delta}},
\]
where $u_i$'s are non-negative integers satisfying $0 \le u_1 < u_2 < \dots < u_{\delta}$. These integers are readily computed, as we will now show.

Evaluating a differential operator $\cd' \in \Cc(t)[\del_t]$ on $t^a$ for some indeterminate $a$ and expanding the result in power series around $t=0$ gives an expression of the form
\begin{equation*} 
 \cd(t^a) = \iota_\cd(a)t^k + \text{higher order terms.}
\end{equation*}

\begin{definition}
  The coefficient of the lowest order term appearing in $\cd'(t^a)$ is a polynomial $\iota_\cd(a)$ which is called the \emph{indicial polynomial of $\cd'$ at $t=0$}.
\end{definition}

\begin{lemma}
  The integers $u_1,\dots,u_{\delta}$ are precisely the roots of the indicial equation of $\cd$.
\end{lemma}
\begin{proof}
  It is clear that each $u_i$ must be a root of the indicial equation. There can be no other root since the indicial equation of $p$ has the same degree as $\cd$.
\end{proof}

To start integrating our Picard--Fuchs equation $\cd=0$, we need only compute the following integrals:
\[
  \sigma_{i}^{(k)}(0) = \int_{\gamma_i} \res \omega^{(k)}(0),\, \quad k \in \{u_1,\dots,u_{\delta}\},\, i=1,\dots,m.
\]

\section{The algorithm} \label{sec:algorithm}

Let $X = Z(f_X) \subset \Pp^{n+1}$ be a smooth hypersurface of degree $d$. We want to compute a period matrix for $X$. The algorithm outlined in the introduction gives this period matrix in theory by deforming $X$ to the Fermat hypersurface $Z(\sum_{i=0}^{n+1} x_i^d)$. However, we can get better performance by making full use of the inductive nature of the algorithm. In this section, we will give a more effective strategy and clarify some of the steps that were only sketched in the introduction. 

Roughly, what works well is to approach $X$ by a sequence of hypersurfaces $Z(g_0),\dots,Z(g_{n+1})$ where $g_{r+1}:=f_X$, each of the hypersurfaces $Z(g_i)$ are smooth, each consecutive pair of hypersurfaces $(Z(g_i),Z(g_{i+1}))$ are ``close to each other'' and we start with $Z(g_0)$ of \emph{Fermat type}.

\begin{definition}
  A hypersurface is called of \emph{Fermat type} if its defining equation is a sum of powers with arbitrary coefficients, 
 \[
   c_0 x_0^{d} + \dots + c_{n+1}x_{n+1}^d,
 \]
 where $c_i \in \Cc \setminus 0$ for all $i$.
\end{definition}

The notion of being ``close together'' is measured by the size of the \emph{support} of $g_{i+1}-g_i$ in the following sense.

\begin{definition}
  The \emph{support of a polynomial $p$} is the set of monomials which appear with non-zero coefficient in $p$.
\end{definition}

Once the hypersurfaces $Z(g_i)$ are constructed, the periods of $Z(g_1)$ can be deduced from the periods of the Fermat type hypersurface $Z(g_0)$, and the periods of $Z(g_2)$ deduced from the periods of $Z(g_1)$, etc. We will give heuristics on how to construct a sequence $(g_0,\dots,g_{r+1})$ which improves performance in Section~\ref{sec:good_sequence}. 

We mean by \emph{period homotopy} the process of determining a period matrix for one hypersurface from the period matrix of another hypersurface. We give details on how to perform period homotopy in Section~\ref{sec:period_homotopy}. 

\subsection{Constructing a good sequence}\label{sec:good_sequence} 

In theory, one may apply period homotopy to any sequence of smooth hypersurfaces. However, in practice, some sequences give vastly superior performance over others. We describe our method of constructing a sequence which tends to perform well.

We start with $g_0 = \sum_{i=0}^{n+1}c_i x_i^d$ whose coefficients $c_i \in \Cc \setminus \{0\}$ are designed to match the coefficients of $f_X$ as closely as possible: If the monomial $x_i^d$ appears with non-zero coefficient in $f_X$ then we will take $c_i$ to be this coefficient. Otherwise, we will randomly generate $c_i$ such that its order of magnitude matches the size of the coefficients of $f_X$. The rest of the sequence is constructed inductively, by following the set of heuristics below. 

\begin{namedNum}{Heuristics}
  Guiding heuristics while constructing the polynomial $g_i$ from $g_{i-1}$:
  \begin{enumerate}
    \item The support of $f_X - g_i$ must be strictly contained in the support of $f_X - g_{i-1}$. 
    \item The support of $g_i - g_{i-1}$ should be small. 
    \item The support of $g_i$ should be small. 
  \end{enumerate}
\end{namedNum}

Note that the first heuristic guarantees that the process terminates. The second heuristic makes sure that the derivatives of forms defined over the family $(1-t)g_{i-1}+tg_i$ do not get unnecessarily complicated. The third heuristic is intended to make the partial derivatives of $g_i$ simple, which plays a role in the Griffiths--Dwork reduction. 

These heuristics are crude. There may be different paths from $g_0$ to $f_X$ with the size of the relevant supports equal at each step, but giving very different results in terms of the complexity of the Picard--Fuchs equations (see Definition~\ref{def:order-degree}). As a compromise, we facilitate discovery by randomizing the construction process. 

We implemented a randomized greedy algorithm, whose details will not be given here, but the code is available in the source where the function is called {\tt BreakingPath}\footnote{\url{https://github.com/emresertoz/PeriodSuite/suite.mag}}.

\begin{remark}
  A theoretical investigation of what makes a good sequence should be interesting in its own right and, in particular, would be of immense value for this algorithm. See Section~\ref{sec:examples} for examples demonstrating the impact of making a good choice for a path of deformation.
\end{remark}

\subsection{Period homotopy}\label{sec:period_homotopy}

Let $X=Z(f_X)$ and $Y=Z(f_Y)$ be smooth hypersurfaces in $\Pp^{n+1}$ of degree $d$, with $f_X$ and $f_Y$ reduced. Suppose that $\cp_Y$ is a period matrix of $Y$ computed using an ordered residue basis $\omega_Y:=(\tilde \omega_1,\dots,\tilde \omega_N)$, see Definition~\ref{def:residue_basis} for this notion. The particular homology basis used in the computation of $\cp_Y$ is irrelevant. 

We will now give an algorithm which takes as input the tuple $(f_X,f_Y,\cp_Y,\omega_Y)$ and gives as output a tuple $(\cp_X,\omega_X)$ where $\cp_X$ is a period matrix of $X$ computed using an ordered residue basis $\omega_X$.

It will be convenient to use the following notation throughout, so as to avoid computing pole orders explicitly within the text.

\begin{notation}
  Let $p,f \in \Cc[x_0,\dots,x_{n+1}]$ be homogeneous polynomials such that $\deg p = d \ell  f -n -2$ where $\ell$ is a positive integer. Then we will define the following form:
  \[
    \varphi(p,f) := \frac{p}{f^\ell}\Omega.
  \]
  We linearly extend this map to allow for non-homogeneous $p$.
\end{notation}

\subsection*{Step 1: Finding a simultaneous residue basis} 

Let us recall some notation from Section~\ref{sec:basics}. We define $R$ to be the polynomial ring $\Cc[x_0,\dots,x_{n+1}]$, $R_s \subset R$ the subspace of homogeneous degree $s$ polynomials in $R$. For a homogeneous ideal $I$ in $R$, $I_s$ is the intersection $I \cap R_s$. As in Notation~\ref{not:jacobian} we will use $J(f)$ to denote the Jacobian ideal of a polynomial $f$.

\begin{definition}
  Let $p_1,\dots,p_N \in R$ be polynomials such that the sets of forms $\{\varphi(p_i,f_X) \mid i=1,\dots,N\}$  and  $\{\varphi(p_i,f_Y) \mid i=1,\dots,N\}$ are residue bases for $X$ and $Y$ respectively. Then we will call $\{p_1,\dots,p_N\}$ a \emph{simultaneous residue basis for $X$ and $Y$}.
\end{definition}

Using Proposition~\ref{prop:GD-basis_residues} we can find a simultaneous residue basis for $X$ and $Y$. This requires that we find, for each $\ell = 1, \dots, n+1$, a set of polynomials in $R_{d \ell -n-2}$ which descend to bases in both of the following vector spaces:
\begin{itemize}
  \item $R_{d \ell - n- 2}/J(f_X)_{d\ell -n -2}$,
  \item $R_{d \ell - n- 2}/J(f_Y)_{d\ell -n -2}$.
\end{itemize}
This is a problem in linear algebra and is readily solved. 

\begin{remark}
  Heuristically, a basis consisting of polynomials supported on as few monomials as possible tend to give faster performance during the computation of the differential equations. In general, there is no simultaneous residue basis for $X$ and $Y$ consisting only of monomials. However, a basis consisting of a mixture of monomial and binomial terms can always be found and this is what the function {\tt CompatibleCohomologyBasis} in {\tt PeriodSuite} finds.
\end{remark}

\subsection*{Step 2: Computing the Picard--Fuchs equations}

Consider the family of hypersurfaces $\cx_t \subset \Pp^{n+1}$ defined by $f_t = (1-t) f_Y + tf_X$ with parameter $t \in \Cc$. For each $i=1,\dots,N$ let $\omega_i(t)=\varphi(p_i,f_t)$, where $p_1,\dots,p_N$ form a simultaneous residue basis for $X$ and $Y$.

For each $i=1,\dots,N$, use the construction in Section~\ref{sec:PF} to obtain the Picard--Fuchs equation $\cd_i \in \Cc(t)[\del_t]$ of $\omega_i(t)$.

\subsection*{Step 3: Initial conditions}

With $\cd_i \in \Cc(t)[\delta_t]$ the Picard--Fuchs equation of $\omega_i(t)$, we will now compute the initial conditions for the associated period function ${}^i\sigma(t)$, defined as in Equation (\ref{eq:period_function}).

Let $u_1,\dots,u_k$ be the roots of the indicial equation of $\cd_i$. As explained in Section~\ref{sec:initial_conditions} the initial conditions we need are ${}^i\sigma^{(u_j)}(0)$ for $j=1,\dots,k$. These are the period vectors associated to $\omega_i^{(u_j)}(0)$ on $Y$.

Since we are given a residue basis $\tilde \omega_1,\dots,\tilde \omega_N$ as input, we may express $\omega_i^{(u_j)}(0)$ in terms of these basis elements by working with reduced forms. Fixing $i$ and $j$, suppose now that $a_1,\dots,a_N \in \Cc$ are computed such that: 
\[
  \omega_i^{(u_j)}(0) \equiv a_1 \tilde \omega_1 + \dots + a_N \tilde \omega_N,
\]
with equivalence being modulo exact forms.

We know the associated period vector for each of the forms in the residue basis; they are the rows of $\cp_Y$. Therefore, the linear expression of $\omega_i^{(u_j)}(0)$ computed above allows us to compute the corresponding period vector: 
\[
  {}^{i}\sigma^{(u_j)}(0) = 
  \cp_Y \begin{bmatrix}
    a_1 \\ \vdots \\ a_N
  \end{bmatrix}.
\]
Performing this operation for all $j$ gives us all of the initial conditions required to integrate $\cd_i$.

\begin{remark}
  Of course, when $Y$ is a Fermat type hypersurface, we do not need the period matrix for $Y$. It is faster to use the formula given in Theorem~\ref{thm:initials}. 
\end{remark}

\subsection*{Step 4: Numerical integration}

Construct a path $h:[0,1] \to \Cc$ such that $h(0)=0$ and $h(1)=1$, with the additional constraint that the $\cx_{h(u)}$ is smooth for every $u \in [0,1]$. 

\begin{remark}
  To satisfy the last requirement, we need to determine the values of $t$ for which $\cx_t$ is singular (or, determine a non-disconnecting set containing these singular values). A direct attack would be to compute the singular values of $f_t$ by elimination theory. This methods is \emph{potentially} expensive, but works well in practice. The second method would be to compute the poles of the coefficients of $\cd_i$ for each $i$. This second method is simpler to execute, but gives many more points than just the singular fibers of the family $\cx_t$, making it harder to construct $h$. A Voronoi cell decomposition can be used to find a path that stays as far away from the singularities as possible, improving the time of integration.\footnote{We implemented both path finding schemes in {\tt PeriodSuite}.} 
\end{remark}

We now need to integrate each $\cd_i$ over the path $h$ using the initial conditions computed in Step 3. The result of this integration will form the $i$-th row of the matrix $\cp_X$.

\begin{remark}
  The integration of $\cd_i$ for each $i$ should be performed using the same path $h$. Otherwise, one risks computing each row of $\cp_X$ using a different \emph{homology} basis!
\end{remark}

\subsection*{Output} Return $\cp_X$ and the residue basis $\{\varphi(p_i,f_X)\mid i=1,\dots,N\}$.

\subsection{Classical periods}\label{sec:classical_periods}

We need the period matrix for the intermediate hypersurfaces used to approach $X$ in order to perform period homotopy. However, since we are primarily interested in the Hodge structure on a variety, we do not need all of the period matrix on the target hypersurface $X$. 

\begin{definition}
  The period of a rational form with pole order $\ell \le \ceil{\frac{n}{2}}$ over $X$ will be called a \emph{classical period of $X$}. The portion of the period matrix formed by elements of the residue basis whose reductions have pole order $\ell \le \ceil{\frac{n}{2}}$ will be called a \emph{classical period matrix of $X$}.
\end{definition}

The notation is justified by allusion to curves and surfaces. In the case of curves, classically, the periods that were studied are only those corresponding to pole order $\ell =1$, which is sufficient for geometric applications. In the case of surfaces, again, one only considers pole order $\ell =1$: for instance one studies \emph{the} period vector of a K3 surface rather than its $22\times 21$ period matrix.

When we are computing the periods of a sequence of hypersurfaces with target $X$, it can be much faster to compute only the classical periods of $X$. The modification to the period homotopy need only be made in the following sense: the period matrices of intermediate hypersurfaces have to be computed as usual, but at the step when the periods of $X$ are computed, we determine and integrate only the Picard--Fuchs equations of residue basis elements with pole order $\ell \le \ceil{\frac{n}{2}}$.

\section{Periods of the Fermat hypersurface}  \label{sec:periods_of_fermat}

Let $f_Y = x_0^d + \dots + x_{n}^d - x_{n+1}^d$ and $Y := Z(f_Y) \in \Pp^{n+1}$ be the corresponding \emph{Fermat hypersurface}. The hypersurface $Y$ is simple enough to be studied in depth and, therefore, serves as our base point for period homotopy. We can scale each $x_i^d$ by a non-zero constant and the results of this section will carry through, see Section~\ref{sec:change_type}. 

\subsection{Pham cycles} \label{sec:pham_cycles}

Fix the following root of unity $\xi = e^{\frac{2\pi \sqrt{-1}}{d}}$. Let $G \subset \aut(Y)$ be the sub-group of automorphisms of $Y$ generated by scaling the coordinates by $\xi$. To be more explicit, let $R=\Zz[G]$ be the corresponding group ring and $\Zz[t_0,\dots,t_{n+1}]$ a polynomial ring. We will associate to $t_i$ the following action:
\[
  [x_0,\dots,x_i,\dots,x_{n+1}] \mapsto [x_0,\dots,\xi x_i,\dots,x_{n+1}].
\]
This defines a map
\[
  \Zz[t_0,\dots,t_{n+1}] \to R,
\]
where the kernel is given by the ideal $K$ generated by the following $n+3$ elements:
\[
  t_0\cdots t_{n+1} - 1,\quad t_i^d - 1 ,\, i=0,\dots,n+1.
\]
We will thus make the following identification without further mention:
\[
  \Zz[t_0,\dots,t_{n+1}]/K \simeq R.
\]

Note that we can write $t_{n+1}$ as $(t_0\cdots t_n)\inv = t_0^{d-1} \cdots t_n^{d-1}$ so we can eliminate $t_{n+1}$ when necessary; for instance, when we pass to the affine chart $x_{n+1} \neq 0$ of $\Pp^{n+1}$.

Let $U \subset \Pp^{n+1}$ be the affine chart $x_{n+1} = 1$ and $Y^o:=Y\cap U \subset Y$ be the corresponding locus, i.e., $Y^o = \set{(x_0,\dots,x_n) \mid x_0^d+\dots+x_n^d=1}$.

\begin{definition}
  Let $D:= \set{(s_0,\dots,s_n) \mid s_i \in [0,1] \subset \Rr,\, s_0^d + \dots + s_n^d = 1} \subset Y^o$. 
\end{definition}

\begin{remark}
  The space $D$ is homeomorphic to the $n$-simplex $s_0 + \dots + s_n = 1$. 
\end{remark}

Let $S:= (1-t_0\inv)\cdots(1-t_n\inv)D$, which is to be viewed as an $n$-chain in $Y^o$. By a straightforward computation, one can check that the boundary $\del S$ is zero, thus $[S] \in \H_{n}(Y^o,\Zz)$. Since $G$ acts on $Y$ as well as $Y^o$, we may view $\H_n(Y^o,\Zz)$ and $\H_n(Y,\Zz)$ as $R = \Zz[G]$ modules via the induced covariant action on homology.

\begin{notation}
  Define the $R$-module homomorphism $\varphi: R \to \H_{n}(Y^o,\Zz)$ given by $1 \mapsto [S]$.
\end{notation}

\begin{theorem}[Theorem 1~\cite{pham--fermat}]\label{thm:pham}
  The morphism $\varphi$ is surjective, with kernel given by:
  \[
    \ker \varphi = (1+t_i+\dots+t_i^{d-1} \mid i=0,\dots,n).
  \]
\end{theorem}

\begin{remark}\label{rem:middle_dimension}
  By standard arguments, one can show that the middle homology $\H_n(Y,\Zz)$ of a degree $d$ hypersurface in $\Pp^{n+1}$ is torsion-free and of rank  
  \[
    \sum_{i=0}^{n} (-1)^i {n+2 \choose i} d^{n+1-i} + (-1)^{n+1} 2 \left\lceil \frac{n}{2} \right\rceil.
  \]
\end{remark}

Any cycle in $\gamma \in \H_{n}(Y^o,\Zz)$ can be viewed as a cycle in $\H_{n}(Y,\Zz)$, but more is true: $\gamma$ will not intersect the hyperplane class. Therefore, the natural inclusion map $Y^o \to Y$ defines a map $\xj:\H_{n}(Y^o,\Zz) \to \PH_{n}(Y,\Zz)$. Furthermore, $\xj$ is surjective since any primitive cycle in $Y$ can be arranged so as not to intersect the hyperplane section $x_0 =0$.

We will need the following explicit description of the primitive cohomology. We learned this statement from an earlier draft of~\cite{degtyarev-16} but the final version no longer contains it. We take the liberty to repeat it here and give a different proof.

\begin{notation}
  Let $J$ be the ideal of $R$ generated by $(1+t_i+\dots+t_i^{d-1} \mid i=0,\dots,n+1)$. Notice that $J$ has one more generator than the kernel of $\varphi$ given in Theorem~\ref{thm:pham}.
\end{notation}

\begin{definition}
  Let $M$ be a $\Zz$-module. A submodule $N$ in a module $M$ is called \emph{primitive} if $M/N$ is torsion-free. The \emph{primitive closure} (or \emph{saturation}) of $N$, denoted $\pcl(N)$, is the smallest primitive submodule containing $N$. 
\end{definition}

\begin{corollary}\label{cor:primitive_basis}
  Composing the map $\varphi$ with the surjection $\xj:\H_{n}(Y^o,\Zz) \tos \PH_{n}(Y,\Zz)$ defines a surjective map $R \to \PH_{n}(Y,\Zz): 1 \mapsto [S]$. The kernel of this map is 
  \[
    \ker (\xj \circ \varphi) = \pcl(J),
  \]
  where $\pcl$ denotes the primitive closure of the ideal in $R$ viewed as a $\Zz$-submodule. 
\end{corollary}

\begin{proof}
  The morphism $R \to \PH_n(Y,\Zz)$ is $G$-equivariant by design. Therefore its kernel must be a $G$-invariant ideal. Combining this observation with Theorem~\ref{thm:pham} we conclude that the kernel contains $J$. On the other hand, $\PH_n(Y,\Zz)$ is torsion-free which means that the kernel must contain the primitive closure of $J$.

  To show that the kernel cannot contain anything else we simply compute the dimension of $\PH_n(Y,\Qq)$ using the formula given in Remark~\ref{rem:middle_dimension} and see that it equals the dimension of $(R/\pcl(J)) \otimes_{\Zz} \Qq$.
\end{proof}

\subsection{Making the residue map explicit} \label{sec:compute_residues}

Pick $\alpha := (a_0,\dots,a_{n+1}) \in \Zz_{> 0}^{n+2}$ such that $|\alpha| = \ell d$ for an integer $\ell$. Recall $f_Y = x_0^d + \dots + x_n^d - x_{n+1}^d$ and $Y = Z(f_Y) \subset \Pp^{n+1}$. We define:
\begin{equation}\label{eq:monomial}
  \omega_\alpha := x_0^{a_0-1}\dots x_{n+1}^{a_{n+1}-1} \frac{\Omega}{f_Y^\ell}.
\end{equation}
Recall Griffiths' residue map from Section~\ref{sec:griffiths_residues}:
\[
  \res: \bigoplus_{\ell \ge 1} \H^0(\omega_{\Pp^{n+1}}(\ell [Y])) \to \PH^n(Y,\Cc).
\]
In this section we will make this residue map explicit, by associating to $\omega_\alpha$ a meromorphic form $\eta_\alpha$ on $Y$. The poles of $\eta_\alpha$ will be positioned along a hyperplane section of $Y$, so that $\eta_\alpha$ represents a primitive cohomology class on $Y$. 

\begin{remark}
  A general treatment of the residue map appears in~\cite{carlson-80} from which explicit formulas like the ones below may be derived algebraically.
\end{remark}

\begin{lemma}\label{lem:residue}
  Let $\lambda \in \Cc \setminus \{0\}$ be a non-zero number, $a \in \Zz_{>0}$ and $0 < \varepsilon \ll 1$. Then, we have the following residue formula:
  \begin{equation*}
    \frac{1}{2\pi i} \int_{|z-\lambda| = \varepsilon} \frac{z^{a-1}\dd z}{(\lambda^d-z^d)^\ell} = \frac{(-1)^\ell}{d^\ell (\ell-1)!} \prod_{j=1}^{\ell-1} (a-jd)\lambda^{a-\ell d}.
  \end{equation*}
\end{lemma}
\begin{proof}
  Since $\lambda$ is non-zero and $\varepsilon$ is small, we can choose a branch of the $d$-th root function such that the substitution $z=\lambda u ^{1/d}$ is well defined, where $u$ is restricted to a neighbourhood of 1. After performing this substitution, the integrand obtains a form suitable for applying Cauchy's differentiation formula and the result follows\footnote{We thank the \emph{mathoverflow} user \emph{GH from MO} for spotting this simple proof.}.
\end{proof}

\begin{notation}
  Let $c_{d,\ell,a}$ be the value of the integral given in Lemma~\ref{lem:residue} with $\lambda=1$.
\end{notation}

\begin{definition}\label{def:eta}
  A tuple $\alpha = (a_0,\dots,a_{n+1}) \in \Zz^{n+2}_{>0}$ satisfying $\sum_{i=0}^{n+1} a_i \equiv 0 \imod d$ will be called an \emph{admissible index}. For an admissible index $\alpha$ we define the form:
  \[
    \eta_{\alpha} := \left.\left( x_0^{a_0-1}\cdots x_{n-1}^{a_{n-1}-1} x_{n}^{a_n-d} \dd x_0 \cdots \dd x_{n-1} \right) \right|_Y.
  \]
\end{definition}

\begin{proposition}\label{prop:explicit_residue}
  For an admissible index $\alpha$, we have an equivalence modulo exact forms:
  \begin{equation*}
    \res \omega_\alpha \equiv c_{d,\ell,a_{n+1}}\eta_\alpha.
  \end{equation*}
\end{proposition}
\begin{proof}
  In order to end up with poles that are favorably positioned, we will begin working in the affine chart $U_n := \{x_n \neq 0\} \subset \Pp^{n+1}$ and later pass to the affine chart $U_{n+1} := \{x_{n+1} \neq 0\} \subset \Pp^{n+1}$. Let $u_i = \frac{x_i}{x_n}$ denote the coordinate functions on $U_n$. We then have: 
  \[
    \omega_\alpha|_{U_n} = \frac{u_0^{a_0-1}\cdots u_{n-1}^{a_{n-1}-1}u_{n+1}^{a_{n+1}-1}}{(u_0^d+\dots+u_{n-1}^{d}+1-u_{n+1}^d)^\ell}\dd u_0\dots \dd u_{n-1} \dd u_{n+1}.
  \]

  On the open locus where $u_{n+1}\neq 0$, the tangent vector $\frac{\del}{\del u_{n+1}}$ does not belong to the tangent space of $Y$. Therefore we can construct a tubular neighbourhood around $\tilde Y := Y \cap \{x_n x_{n+1} \neq 0\}$ by using an $S^1$-fibration $\pi:\tau(\tilde Y) \to \tilde Y$ where $\tau(\tilde Y) \subset \Pp^{n+1} \setminus Y$ and each fiber of $\pi$ is a (small) circle in the direction of $\frac{\del}{\del u_{n+1}}$.

Integrating $\omega_\alpha$ along this $S^1$-fibration, we can push it down to a form on $\tilde Y$ which we will denote by $\pi_! \omega_\alpha$. Fix a point $p=[u_0,\dots,u_{n-1},1,u_{n+1}] \in Y$ and notice that $u_{n+1}^d = u_0^d + \dots + u_{n-1}^d +1$. The integral over the fiber of $\pi:\tau(\tilde Y) \to \tilde Y$ over $p$ will give:
\[
  \pi_! \omega_{\alpha}|_p = u_0^{a_0-1}\cdots u_{n-1}^{a_{n-1}-1} \left(\int_{|z-u_{n+1}| < \varepsilon} \frac{z^{a_{n+1}-1} \dd z}{(u_{n+1}^d - z^d)^\ell}\right) \dd u_0 \dots \dd u_{n-1},
\]
for small $\varepsilon >0$. Now we use Lemma~\ref{lem:residue} to conclude:
\[
\pi_! \omega_\alpha|_p = c_{d,\ell,a_{n+1}} u_0^{a_0-1}\cdots u_{n-1}^{a_{n-1}-1} u_{n+1}^{a_{n+1}-\ell d}\dd u_0 \dots \dd u_{n-1}.
\]

Due to the severity of the poles at $u_{n+1} =0$ we change coordinates to the affine chart $U_{n+1} := \{x_{n+1} \neq 0\} \subset \Pp^{n+1}$. Let us denote by $v_i=\frac{x_i}{x_{n+1}}$ the coordinate functions on $U_{n+1}$. Naturally, $v_{n+1} = 1$ and $v_n u_i = v_i$ for $i=0,\dots,n+1$. By mere substitution, and using the identity $a_0 + \dots + a_{n+1} = \ell d$, we have:
\[
  u_0^{a_0-1}\cdots u_{n-1}^{a_{n-1}-1}u_{n+1}^{a_{n+1}-\ell d} = v_0^{a_0-1}\cdots v_{n-1}^{a_{n-1}-1}v_{n}^{a_n + n}.  
\]
On the other hand, a straightforward computation reveals that the volume form changes in the following way:
\[
  \dd u_0 \dots \dd u_{n+1}  = v_n^{-n-d} \dd v_0 \cdots \dd v_{n-1}, 
\]
where we used both of the following identities on the hypersurface $Y$: $v_0^d + \dots + v_{n}^d = 1$ and $\dd v_n = \frac{v_0^{d-1}}{v_{n}^{d-1}}\dd v_0 + \dots + \frac{v_{n-1}^{d-1}}{v_n^{d-1}}\dd v_{n-1}$.

Putting these two together gives us the desired result. Note that the tubular neighbourhood gets infinitely thin as we approach the locus $v_n=0$. Nevertheless, we can extend our representation of $\pi_!\omega_\alpha$ to the locus $v_n =0$ simply because it has no poles there, as can be verified by substituting $\dd v_i = \frac{v_n^{d-1}}{v_i^{d-1}}\dd v_n + \dots$ for any one of the $v_i$ appearing in the volume form.
\end{proof}

Combining Section~\ref{sec:griffiths_residues} with Proposition~\ref{prop:explicit_residue} we conclude that a generating set for the primitive cohomology $\PH^n(Y,\Cc)$ is given by the set of forms $\eta_\alpha$ as $\alpha$ ranges through admissible indices.

\subsection{Residues over Pham cycles}

Recall from Section~\ref{sec:pham_cycles} that $D = \set{(s_0,\dots,s_n) \mid s_i \in [0,1] , s_0^d + \dots + s_n^d =1} \subset Y^o$ and $S = (1-t_0\inv)\cdots(1-t_n\inv) D$ is the Pham cycle generating the primitive integral homology $\PH_n(Y,\Zz)$ under the action of the group ring $R = \Zz[t_0,\dots,t_{n+1}]$.

\begin{notation}
  For a given integer sequence $\beta=(\beta_0,\dots,\beta_{n+1}) \in \Zz^{n+2}$ denote by $t^\beta$ the element $t_0^{\beta_0}\cdots t_{n+1}^{\beta_{n+1}}$ of $R$.  
\end{notation}

Then for any $\beta$ and any admissible $\alpha$ we wish to compute the integral:
\[
  \sigma_{\alpha,\beta}:=\int_{t^\beta S} \eta_\alpha,
\]
where $\eta_\alpha$ is given in Definition~\ref{def:eta}.

Let $\Gamma(x) = \int_0^\infty e^{-t} t^{x-1} \dd t$ denote the Euler gamma function. Recall the identity (\cite[pg. 19]{artin--gamma}):
\begin{align} \label{eqn:euler_beta}
  \int_0^1 t^{a-1}(1-t)^{b-1} \dd t &= \frac{\Gamma(a)\Gamma(b)}{\Gamma(a+b)},\quad a,b \in \Rr_+.
\end{align}

\begin{lemma} \label{lem:integral}
  The integral of $\eta_\alpha$ over $D$ has the following value:
  \[
    \int_D \eta_\alpha = \frac{\prod_{i=0}^n \Gamma(\frac{\alpha_i}{d})}{d^{n}\Gamma(\sum_{i=0}^n \frac{\alpha_i}{d})}.
  \]
\end{lemma}
\begin{proof}
  For $s \in [0,1] \subset \Rr$ the positive $d$-th root of $s$ is well defined and will be denoted by $s^{\frac{1}{d}}$. We will make use of this fact without further mention in the following substitutions.

  With the understanding that empty product equals 1, and for real coordinates $(s_0,\dots,s_n) \in D$ we introduce the following variables inductively:
  \[
    u_i := \frac{s_i^d}{\prod_{j=0}^{i-1} (1-u_j)},\quad i=0,\dots,n-1.
  \]
  We now have:
  \begin{align*}
    s_i &= u_i^{\frac{1}{d}}\prod_{j=0}^{i-1} (1-u_i)^{\frac{1}{d}},  \quad i=0,\dots,n-1.
  \end{align*}
  By induction on $n$ one can see that:
  \[
    s_n^d=1- s_0^d - \dots - s_{n-1}^d = \prod_{i=0}^{n-1}(1-u_i). 
  \]
  The differentials $\dd s_i$ can be expressed in terms of the differentials $\dd u_j$ for $j \le i$ in the following manner:
  \[
    \dd s_i = \frac{1}{d} u_i^{\frac{1}{d}-1}\prod_{j=0}^{i-1}(1-u_j)^{\frac{1}{d}} \dd u_i + \dots,
  \]
  where we omitted the terms involving $\dd u_j$'s for $j < i$ since they will be killed when taking the volume form. Indeed, the volume form can now be written as:
  \[
    \dd s_0 \cdots \dd s_{n-1} = \frac{1}{d^n} \left(\prod_{i=0}^{n-1} u_i^{\frac{1}{d}-1} (1-u_i)^{\frac{n-i-1}{d}} \right)\dd u_0 \cdots \dd u_{n-1}. 
  \]
  Putting everything together, we get:
  \[
    \eta_\alpha|_D = \frac{1}{d^n} \prod_{i=0}^{n-1} \left( u_i^{\frac{\alpha_i}{d}-1} (1-u_i)^{\frac{\sum_{j=i+1}^{n} \alpha_j}{d}-1}\right) \dd u_0 \cdots \dd u_{n-1}.
  \]
  We have successfully separated the variables in the integrand, therefore we have the following equality:
  \[
    \int_D \eta_\alpha = \frac{1}{d^n} \prod_{i=0}^{n-1} \left(\int_0^1 u_i^{\frac{\alpha_i}{d}-1} (1-u_i)^{\frac{\sum_{j=i+1}^n \alpha_j}{d} -1 } \dd u_i\right).
  \]
  Substituting the identity in Equation~\ref{eqn:euler_beta} we get a telescoping product, giving us the claimed result.
\end{proof}

\noindent For integer sequences $\alpha=(\alpha_0,\dots,\alpha_{n+1})$, $\beta=(\beta_{0},\dots,\beta_{n+1})$ let $\alpha\cdot \beta = \sum_{i=0}^{n+1} \alpha_i \beta_i$.

\begin{lemma}\label{lem:pullback}
  For admissible $\alpha$ and any $\beta$ we have:
  \[
    \int_{t^\beta D} \eta_\alpha = \xi^{\alpha \cdot \beta} \int_D \eta_\alpha.
  \]
\end{lemma}
\begin{proof} This follows at once by the change of variables $(t^\beta)^* \eta_\alpha$.
\end{proof}

\begin{proposition}\label{prop:periods}
  For any admissible $\alpha$ and any $\beta$ we have:
  \[
    \int_{t^\beta S} \eta_\alpha = \xi^{\alpha\cdot \beta} \left( \frac{\prod_{i=0}^n \left( (1-\xi^{-\alpha_i})\Gamma(\frac{\alpha_i}{d})\right)}{d^{n} \Gamma(\sum_{i=0}^n \frac{\alpha_i}{d})}\right).
  \]
\end{proposition}
\begin{proof}
  Recalling $S = (1-t_0\inv)\cdots(1-t_n\inv) D$, this follows from Lemmas~\ref{lem:integral} and~\ref{lem:pullback}.
\end{proof}

\begin{remark}
  This formula also appears in~\cite[\S I.7]{dmos-82} and~\cite[\S 15.2]{movasati--hodge}.
\end{remark}

We recall here that for an integer sequence $\alpha = (a_0,\dots,a_{n+1}) \in \Zz^{n+2}_{>0}$ such that $|\alpha| = \ell d$ we defined the meromorphic form
\begin{equation}\label{eq:omega_alpha}
  \omega_\alpha = \frac{x_0^{a_0-1}\cdots x_{n+1}^{a_{n+1}-1}}{(x_0^d + \dots x_n^d - x_{n+1}^d)^\ell} \Omega,
\end{equation}
where $\Omega$ is the volume form on $\Pp^{n+1}$ defined in Equation (\ref{eq:Omega}). Combining Proposition~\ref{prop:periods} with Proposition~\ref{prop:explicit_residue} we now obtain the following result.

\begin{theorem}\label{thm:initials}
  Let $\alpha,\beta \in \Zz_{>0}^{n+2}$ where $|\alpha| = \ell d$ for some $\ell \in \Zz_{>0}$. Then we have:
  \[
    \int_{t^\beta S} \res \omega_\alpha =  -\prod_{j=1}^{\ell-1}\left(1-\frac{a_{n+1}}{jd}\right) \prod_{i=0}^n\left( \frac{1-\xi^{-\alpha_i}}{d} \right)  \frac{\prod_{i=0}^n \Gamma(\frac{\alpha_i}{d})}{\Gamma(\sum_{i=0}^n \frac{\alpha_i}{d})} \xi^{\alpha\cdot \beta}.
  \]
\end{theorem}

\subsection{Changing the type of the Fermat curve}  \label{sec:change_type}

For a sequence of non-zero complex numbers $\xc:=(c_0,\dots,c_{n+1})$ we may want to take our initial Fermat hypersurface to be cut out by:
\[
  f_{\xc} := c_0x_0^{d} + \dots + c_{n+1}x_{n+1}^d.
\]
Then the integrals given in Theorem~\ref{thm:initials} change slightly, which we record here for convenience.

Our standard choice of Fermat hypersurface is $f_{(1,\dots,1,-1)}$ which we will continue to denote by $f_Y$. For $i=0,\dots,n$ let $\mu_i \in \Cc$ be a $d$-th root of $c_i$ and let $\mu_{n+1} \in \Cc$ be a $d$-th root of $-c_{n+1}$. Defining $\varphi: \Pp^{n+1} \to \Pp^{n+1}: [x_0,\dots,x_{n+1}] \mapsto [\mu_0\inv x_0 , \dots, \mu_{n+1}\inv x_{n+1}]$, we get an isomorphism:
\[
  \varphi|_{Z(f_Y)} : Z(f_Y) \isoto Z(f_\xc).
\]
We will use the basis of $\PH_n(Z(f_\xc),\Zz)$ obtained as the image of the basis we just constructed for $\PH_n(Y,\Zz)$ using Pham cycles. 

Let us write $\omega_{\alpha,\xc}$ for the rational form defined as in (\ref{eq:omega_alpha}) but with the denominator $f_\xc^{\ell}$ instead of $f_Y^\ell$. Then the result of Theorem~\ref{thm:initials} need only be modified by scaling the integration via $\prod_{i=0}^{n+1} \mu_i^{-a_i}$: 
\begin{equation}\label{eq:modified_init}
  \int_{\varphi(\gamma)}\res \omega_{\alpha,\xc} = \int_\gamma \res \varphi^* \omega_{\alpha,\xc} = \left( \prod_{i=0}^{n+1} \mu_i^{-a_i}  \right) \int_\gamma \res\omega_\alpha.
\end{equation}

\subsection{From primitive to entire homology} \label{sec:primitive_to_entire}

When $n$ is odd, homology agrees with primitive homology. However, when $n$ is even we need to add an extra class to primitive homology to generate homology. For this section assume $n$ is even.

As demonstrated in Lemma~\ref{lem:generators} we may add the class of a linear space contained in $Y$. We picked $L$ introduced as the image of the map in Equation (\ref{eq:L}) for our implementation. However, the rest of this section is independent of which linear space is chosen.

Let us denote by $[L]$ the homology class of $L$. Furthermore, let us denote by $\gamma_\beta$ the homology class of $t^\beta S$ in $\PH_n(Y,\Zz)$. Using Corollary~\ref{cor:primitive_basis}, pick a subset $B \subset \Zz^{n+2}$ such that $\{\gamma_\beta \mid \beta \in B\}$ freely generates $\PH_n(Y,\Zz)$.

As mentioned in Remark~\ref{rem:v}, with $h$ defined as in Notation~\ref{not:h}, we may write $[L] = \frac{1}{d}h + \gamma_L$ where $\gamma_L \in \PH_n(Y,\Qq)$. Since the residues $\res \omega_\alpha$ are in the primitive cohomology, they annihilate the class $h$. Therefore, we have the equality:
\[
  \int_{[L]} \res \omega_\alpha = \int_{\gamma_L} \res \omega_\alpha.
\]
In particular, the period of $L$ can be computed without taking any new integrals. We need only determine the expression of $\gamma_L$ in terms of $\gamma_\beta$ for $\beta \in B$.

Write $\gamma_L = \sum_{\beta \in B} a_{\beta}\gamma_\beta$ for $a_\beta \in \Qq$ and let $a=[a_\beta]_{\beta \in B}$ be the row vector formed by these coefficients. Let $b = [\gamma_L \cdot \gamma_\beta]_{\beta \in B}$ be the row vector formed by intersecting $L$ with the basis elements, note that $\gamma_L \cdot \gamma_\beta = [L] \cdot \gamma_\beta$ for all $\beta \in B$. Finally, let $M=[\gamma_\beta \cdot \gamma_{\beta'}]_{\beta,\beta' \in B}$ be the intersection matrix on primitive cohomology. Then, evidently, we have:
\begin{equation}\label{eq:ab}
  a = b \cdot M\inv.
\end{equation}
It remains to compute the various intersection products. For the intersection product on Pham cycles, see any one of~\cite{arnold-1984, movasati--hodge, looijenga-10}. The intersection product of Pham cycles with linear spaces are computed in~\cite{degtyarev-16}.

\begin{remark}
  If we deform the Fermat hypersurface $Y$ to another hypersurface $X$ and carry the basis of homology $\{\gamma_\beta \mid \beta \in B\} \cup \{[L]\}$ during this deformation, then the linear relation computed in Equation (\ref{eq:ab}) remains constant. Therefore, the periods of deformations of $L$ can be computed simply from the periods of deformations of $\gamma_\beta$'s.
\end{remark}

\section{Examples}\label{sec:examples}

In this section we will demonstrate the performance of our algorithm and compare it with other algorithms in the cases of curves. We will also take this opportunity to show how complexity of the problem changes when the path of deformation is broken down into small, preferably monomial, steps (see also Section~\ref{sec:algorithm}).

The implementation of our algorithm is called {\tt PeriodSuite} and is available here\footnote{\url{https://github.com/emresertoz/PeriodSuite}}. All running times refer to the CPU time on a laptop (MacBook Pro 2016, 2.6 GHz Intel Core i7).

\begin{remark}
  We recall here that for our target hypersurfaces we will only compute as many periods as necessary to determine the Hodge structure. We called these \emph{classical periods} and explained the distinction in Section~\ref{sec:classical_periods}.
\end{remark}

\subsection{Comparing period matrices of curves}\label{sec:compare_periods}
We will compare the period matrices we get from {\tt PeriodSuite} against the ones we get from the {\tt periodmatrix} function in the package {\tt algcurves}~\cite{deconinck2011} of Maple. 

Both {\tt algcurves} and {\tt PeriodSuite} use the same basis of differential forms, at least up to permutation. But we do not have control over the homology bases used to compute the period matrices. Thus, when we run either algorithm we find two approximate period matrices $M$ and $N$ which appear very different but for which we expect to have an invertible $2g \times 2g$ integer matrix $B$ such that $MB=N$. A standard application of LLL allows us to find small integer matrices that approximately solve this system. 
If the entries in the matrix $B$ are very small compared to the precision available in $M,N$; and if the matrix $MB$ is very close to $N$, we may be quite sure that the period matrices approximated by both programs are the same after a change of basis in homology by the matrix $B$.

The function {\tt ChangeHomologyBasis} in {\tt PeriodSuite} implements the procedure outlined above. We use this implementation in the examples below where the computation of the change of homology basis $B$ is practically instant.

\begin{remark}
  We compare our algorithm against {\tt algcurves}~\cite{deconinck2011} because, at the time of writing, we were not aware of {\tt RiemannSurfaces}~\cite{nils-18} and {\tt abelfunctions}~\cite{christopher16} did not work in SageMath~8.1. Unfortunately, as numerical integration is slow in Maple, {\tt algcurves} is the slowest of the three giving us an unfair advantage. Nevertheless, we feel the comparison serves well to illustrate the weaknesses and strengths of our algorithm.
\end{remark}

\subsection{Elliptic curves} We will begin with smooth plane cubics so that we may display the Picard--Fuchs equations we get. For higher genera, they will be too big to print. The distinction between a straight deformation path and a deformation path broken into monomial steps is well illustrated by the display size of the corresponding ODEs.

\subsection*{Random sparse cubic} Let $f_1=-5x^3 - 2xz^2 + y^3 + 7yz^2$, which defines an elliptic curve with $j$-invariant 
\[
  -\frac{10536960}{323761}. 
\]
We deform periods from the Fermat curve $f_0 = -5x^3+y^3+z^3$. If we deform $f_0$ to $f_1$ directly, we need to integrate the following Picard--Fuchs equation:
\begingroup
\fontsize{8pt}{\baselineskip}\selectfont
\begin{multline*}
  D^2+\frac{3t^8-\frac{5154331}{323761}t^7+\frac{74810025}{5180176}t^6+\frac{721635}{647522}t^5-\frac{4510245}{5180176}t^4+ \frac{75510}{323761}t^3-\frac{2025}{5180176}t^2+\frac{2025}{647522}t-\frac{6075}{5180176}}{t^9-\frac{2564243}{647522}t^8+\frac{14301993}{5180176}t^7+ \frac{308115}{647522}t^6-\frac{2016975}{5180176}t^5+\frac{33705}{323761}t^4+
  \frac{70875}{5180176}t^3-\frac{2025}{323761}t^2+\frac{6075}{5180176}t}D+\\
  \frac{\frac{3}{4}t^7-\frac{12924595}{2590088}t^6+\frac{26678799}{5180176}t^5+\frac{1034475}{5180176}t^4- \frac{90795}{647522}t^3+\frac{80685}{2590088}t^2-\frac{6075}{5180176}t+\frac{2025}{5180176}}{t^9-\frac{2564243}{647522}t^8+\frac{14301993}{5180176}t^7+\frac{308115}{647522}t^6-\frac{2016975}{5180176}t^5+\frac{33705}{323761}t^4+\frac{70875}{5180176}t^3-\frac{2025}{323761}t^2+\frac{6075}{5180176}t},
\end{multline*}
\endgroup
which takes 0.01 seconds to find using our {\tt PeriodSuite} and 0.4 seconds to integrate with {\tt oa--analytic} to 20 digits. The resulting period vector $M_1$ is given below, truncated to 10 digits:
\begin{equation*}
M_1 \sim  [ 0.2547540432 - 0.4890903559\sqrt{-1}, 0.2547540432 + 0.4890903559\sqrt{-1} ].
\end{equation*}
The $j$-invariant of this period vector agrees to $16$ digits with the exact $j$-invariant of the curve.

Alternatively, we can define the following sequence of curves:
\begin{align*}
g_0 &= -5x^3+y^3+z^3 \\
g_1 &= -5x^3-2xz^2+y^3+z^3\\
g_2 &= -5x^3-2xz^2+y^3\\
g_3 &= -5x^3-2xz^2+y^3+7yz^2.
\end{align*}
which changes only one monomial at a time. Let us write $G_i = (1-t)g_{i-1} + t g_i$ for $i=1,2,3$. In general, breaking the path like this results in a significant improvement in speed. In this case, there is not much room for improvement but the ODEs are noticeably simpler.  Let us recall from Section~\ref{sec:classical_periods} that we need to compute the entire period matrix for each of the intermediate curves, but not for the last curve. Therefore, there are two ODEs to integrate for each of $G_1$ and $G_2$ but we need only integrate one for $G_3$: 
\begin{center}
\begin{tabular}[]{LLL}
  G_1: & D + \frac{16t^2}{32t^3 + 135} & D^2 + \frac{112t^3 - 270}{32t^4 + 135t}D \\[5pt]
G_2: & D + \frac{45t - 45}{135t^2 - 270t + 167} & D^2 + \frac{360t^2 - 720t + 328}{135t^3 - 405t^2 + 437t - 167}D \\[5pt]
G_3: & D^2 + \frac{5145t^2}{1715t^3 - 8}D + \frac{5145t}{6860t^3 - 32} &
\end{tabular}.
\end{center}
It takes a second in total to compute these five ODEs and to integrate them. This is slower than the first approach by half a second, due to general overhead, but it is clear from the nature of the ODEs that the second approach will be favorable for more complicated examples. The result of the integration gives a 20 digit period vector $M_2$, which is given below truncated to 10 digits:
\begin{equation*}
  M_2 \sim [ 0.2547540432 + 0.4890903559\sqrt{-1}, -0.2547540432 + 0.4890903559\sqrt{-1}].
\end{equation*}
Note that $M_2$ is slightly different from $M_1$ due to monodromy. Finally, {\tt algcurves} takes a little over 3 seconds to compute $M_3$ to 20 digits, which is given below truncated to 10 digits:
\begin{equation*}
  M_3\sim [ -0.5095080865,-0.2547540432-0.4890903559\sqrt{-1}].
\end{equation*}
Using {\tt ChangeHomologyBasis} we find matrices $B_{12}$ and $B_{13}$ such that $M_1 \sim M_2 B_{12}$ and $M_1 \sim  M_3 B_{13}$ with error less than $10^{-18}$, where: 
\[
  B_{12} = \begin{bmatrix} 0 &  1 \\ -1 & -1 \end{bmatrix} \qquad B_{13} = \begin{bmatrix} -1 &  0 \\ 1 & -1 \end{bmatrix}.  
\]

\subsection*{Random cubic curve} We state the computation times for a randomly generated cubic: 
\[
  f_1:= 4x^3 + 5x^2y + 4x^2z - 7xy^2 + 4xyz + 7xz^2 - 8y^3 - 4yz^2 + 3z^3.
\]
We take $f_0 = 4x^3-8y^3+3z^3$ as the initial point of our deformation and compute the periods to 20 digits of precision. 
The straight path from $f_0$ to $f_1$ took 0.15 seconds for the computation of the ODEs and 2.45 seconds for the integration. Deforming one monomial at a time took 0.70 seconds for the ODEs and 5 seconds for the integration. It takes {\tt algcurves} 6 seconds to compute the periods. In all cases, we get period vectors that are equivalent by inspection and the $j$-invariants agree with the $j$-invariant of the curve.

\subsection{Plane quartic curves}

The ODEs we get from quartic curves are far more difficult then the ones we get from cubic curves and, in particular, they cannot be displayed here. One can, however, roughly measure the complexity of an ODE using two integers: its order and degree.

\begin{definition}\label{def:order-degree}
  Let $\cd \in \Qq[t][\del_t]$ be an ODE with relatively prime coefficients. The \emph{order} of $\cd$ is the highest power of the differentiation operator $\del_t$ appearing in $\cd$. The \emph{degree} of $\cd$ is the highest degree of all the polynomials appearing as coefficients of $\cd$. For $\cd \in \Qq(t)[\del_t]$, an ODE with rational function coefficients, define its order and degree by reducing to the previous case through the clearing of denominators and common factors.
\end{definition}

In practice, the complexity of the Picard--Fuchs equations we encounter, that is, the time it takes to construct as well as to integrate them appear to be reflected well by their order and degree.

\subsection*{Favorable quartic} Consider the following randomly generated sparse quartic:
\[
  f_1=4x^4 + 5xz^3 + 5y^4 - y^3z - 6z^4.
\]
We pick the Fermat curve $f_0 = 4x^4+5y^4-6z^4$ as our initial curve. We begin with the straight path from $f_0$ to $f_1$. The three ODEs for this deformation are computed in 1.58 seconds. The three order and degree pairs are displayed below:
\[
\begin{array}[]{ccc}
 (6,85)&(6,79)&(6,84).
\end{array}
\]
The system is integrated in 50 seconds with end result having 30 digits of precision.

Let us now demonstrate how deforming one monomial at a time fares in comparison. We take the following two step deformation: 
\begin{align*}
g_0 &= 4x^4 + 5y^4 - 6z^4 \\
g_1 &= 4x^4 + 5xz^3 + 5y^4 - 6z^4 \\
g_2 &= 4x^4 + 5xz^3 + 5y^4 - y^3z - 6z^4.
\end{align*}
As explained in Section~\ref{sec:classical_periods}, we have to solve for 6 ODEs for the first family and 3 for the second one. It takes 0.40 seconds to compute these 9 ODEs, as each one of them is simpler than the three ODEs we got previously. With $G_i = (1-t)g_{i-1} + tg_i$ denoting the $i$-th family, we list the order-degree pairs of the corresponding sets of ODEs:
\[
\begin{array}[]{lcccccc}
  G_1: & (2,5)  & (2,4)  & (2,5)  & (3,5) & (2,5) & (3,6) \\
  G_2: & (6,40) & (6,39) & (6,39) &       &       & \phantom{(3,6)}.
\end{array}
\]
We can integrate the entire system in 2.46 seconds with 30 digits of precision. For the sake of space, we will round the result to 4 decimal places:
\[\mbox{\scriptsize
  $\left[\begin{smallmatrix}
    0.1388 - 0.1336\sqrt{-1}&  0.1343 + 0.1384\sqrt{-1}& -0.1433 + 0.1384\sqrt{-1}& -0.1388 + 0.1336\sqrt{-1}& -0.1388 - 0.1432\sqrt{-1}&  0.1388 - 0.1432\sqrt{-1}\\
    0.1285 - 0.1467\sqrt{-1}& -0.1168 + 0.1379\sqrt{-1}&  0.1345 + 0.1379\sqrt{-1}&  0.1256 - 0.1263\sqrt{-1}& -0.1285 - 0.1291\sqrt{-1}& -0.1256 + 0.1430\sqrt{-1}\\
    0.0285 - 0.2047\sqrt{-1}& 0.2052 + 0.0282\sqrt{-1}& 0.1481 + 0.0282\sqrt{-1}&-0.0286 + 0.2048\sqrt{-1}&-0.0285 + 0.1482\sqrt{-1}& 0.0286 + 0.1481\sqrt{-1}
\end{smallmatrix}\right]$
}.
\]

The {\tt algcurves} package cannot handle this curve by default and crashes after an hour of computation. However, if asked to project on to the $y$-axis, instead of the default $x$-axis, it can compute the period matrix in 30 seconds with 30 digits of precision. After permuting the rows so that their choice of holomorphic forms matches ours, and after rounding to 4 decimal places, we get:
\[\mbox{\scriptsize
$\left[\begin{smallmatrix}
-0.1343+0.1384\sqrt{-1}   & -0.0090  & 0.2686   & -0.6806+0.1384\sqrt{-1}    & -0.1388-0.1432\sqrt{-1}    & 0\\
0.1168+0.1379\sqrt{-1}   & 0.0177 & -0.2336 & 0.0934+0.1379\sqrt{-1}  & 0.1256+0.1439\sqrt{-1}   & -0.2541-0.2729\sqrt{-1}\\
-0.2052+0.0283\sqrt{-1} & 0.3533  & 0.4104  & -0.6726+0.0283\sqrt{-1} & -0.0286+0.1481\sqrt{-1} & 0
\end{smallmatrix}\right]$
}.
\]

Applying {\tt ChangeHomologyBasis} we find that the change of basis matrix for homology is:
\[
  B=\begin{bmatrix}
    -2&1&1&3&0&0\\
    0&0&1&1&0&-1\\
    -1&1&0&-1&0&1\\
    0&0&0&-1&0&0\\
    -1&0&0&1&1&1\\
    -1&0&0&2&1&0\end{bmatrix}.
\]
After applying this matrix, the greatest discrepancy between the period matrices is less than~$10^{-18}$. 

\begin{remark}
  To summarize, {\tt PeriodSuite} takes 3 seconds with a monomial path (it takes 52 seconds with a direct path). On the other hand, {\tt algcurves} takes 30 seconds. All computations are done to get 30 digits of final precision. The existence of a small homology change of basis matrix provides overwhelming evidence that the two period matrices agree.
\end{remark}

\subsection*{Time versus precision}
Let us now compare the time it takes for {\tt PeriodSuite} against {\tt algcurves} when we demand higher and higher precision. Figure~\ref{fig:time-precision} plots the time it takes to compute the periods of the ``Favorable quartic'' against requested precision. The blue small dots represent {\tt algcurves} and the red big dots represent {\tt PeriodSuite}. As briefly explained in the introduction, we numerically integrate ODEs which is an operation where time requirement increases linearly with number of digits of precision. On the other hand, {\tt algcurves} numerically computes Riemann integrals, for which cost of computation increases at least quadratically with number of digits of precision~\cite{bailey11}.

\begin{figure}
  \includegraphics[width=0.8\textwidth]{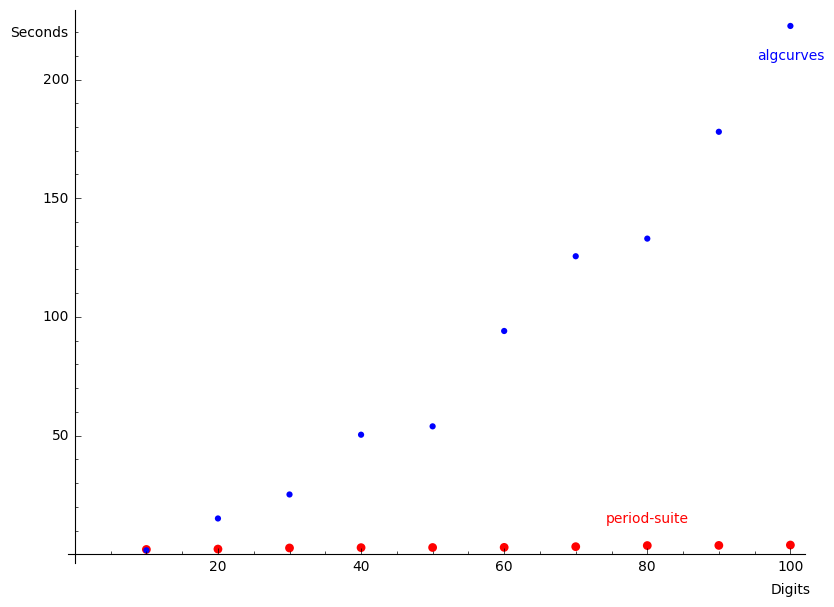}
  \caption{Time vs.\ precision plot for Favorable quartic}
  \label{fig:time-precision}
\end{figure}

\subsection*{Unfavorable quartic} For our algorithm, the time it takes to compute the periods of a curve varies wildly depending on the curve. As a rough measure, the number of monomials that need to be changed from a curve of Fermat type indicates the difficulty of computation. We now look at a random quartic that is much less favorable from this point of view:
\[
  f_1=-7x^3y + 5xy^3 + 7xyz^2 - 4yz^3 + z^4.
\]
A straight path from $f_0 = x^4+y^4+z^4$ is a bad idea. It takes 107 seconds to compute the three Picard--Fuchs equations alone and the resulting equations are of order six and have coefficients of degree 126. We record the three order-degree pairs below for later comparison:
\[
  \begin{array}[]{ccc}
    (6,126) & (6,126) & (6,120).
  \end{array}
\]
Collectively, these differential equations have 360 singularities and the ones near zero are pictured in Figure~\ref{fig:unbroken}. Note that a generically smooth pencil of quadrics can have at most 27 singular fibers. This means, in light of Proposition~\ref{prop:sols_are_holo}, that the rest of the 303 singularities are \emph{apparent} singularities, that is, solutions converge near these points and they do not inhibit computations. Nevertheless, numerical integration took 27 minutes with 100 digits of precision. High precision was required here for successful integration of this complex system. 

\begin{figure}[h]
  \centering
  \includegraphics[width=0.4\textwidth]{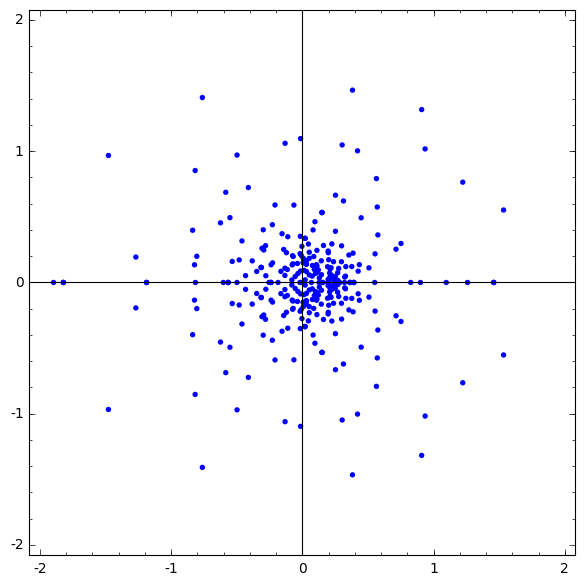}
  \caption{Unfavorable quartic: singularities of the ODEs for the straight path}
  \label{fig:unbroken}
\end{figure}

Let us now break the path of deformation and take the following six steps:
\[
  \begin{aligned}
g_0 &= x^4+y^4+z^4             ,&       g_4 &= x^4-7x^3y+5xy^3-4yz^3+z^4,\\
g_1 &= x^4+5xy^3+y^4+z^4       ,&       g_5 &= -7x^3y+5xy^3-4yz^3+z^4,\\
g_2 &= x^4+5xy^3+z^4           ,&       g_6 &= -7x^3y+5xy^3+7xyz^2-4yz^3+z^4. \\
g_3 &= x^4+5xy^3-4yz^3+z^4     ,&  
\end{aligned}
\]
As before, each deformation is given by $G_i = (1-t)g_{i-1} + tg_i$ where $i=1,\dots,6$. Recall that the first five families require the computation of six ODEs each and the last one requires only three. These 33 ODEs are determined in 42 seconds, already beating the straight path. Furthermore, these ODEs behave far better than the previous three. We list the order-degree pairs for these ODEs below: 
\begin{equation*}
  \begin{array}{ lcccccc } 
    G_1: & ( 2, 5 )  & ( 2, 4 )  & ( 2, 5 )  & ( 3, 6 )  & ( 2, 5 )  & ( 3, 5 ) \\
    G_2 : & ( 2, 3 )  & ( 2, 4 )  & ( 2, 3 )  & ( 3, 5 )  & ( 3, 4 )  & ( 2, 4 ) \\
    G_3 : & ( 6, 17 ) & ( 6, 17 ) & ( 6, 16 ) & ( 7, 18 ) & ( 7, 17 ) & ( 7, 18 ) \\
    G_4 : & ( 6, 36 ) & ( 6, 39 ) & ( 6, 35 ) & ( 7, 41 ) & ( 7, 47 ) & ( 7, 43 ) \\
    G_5 : & ( 6, 20 ) & ( 6, 22 ) & ( 6, 19 ) & ( 7, 22 ) & ( 7, 24 ) & ( 7, 21 ) \\
    G_6 : & ( 6, 29 ) & ( 6, 28 ) & ( 6, 26 ) &           &           &
  \end{array}
\end{equation*}
The singularities of these six systems of ODEs are displayed in Figure~\ref{fig:broken}. Most of the singularities are once again apparent. 
  The integration of the entire system takes less than 3 minutes with 30 digits of precision.
  
On the other hand, {\tt algcurves} does not find this curve any more difficult than the previous one. It can compute the periods to 30 digits in 29 seconds. It may be that our methods are not optimal for curves, especially if low precision period matrices are sufficient. 

\begin{figure}[h]
  \includegraphics[width=0.3\textwidth]{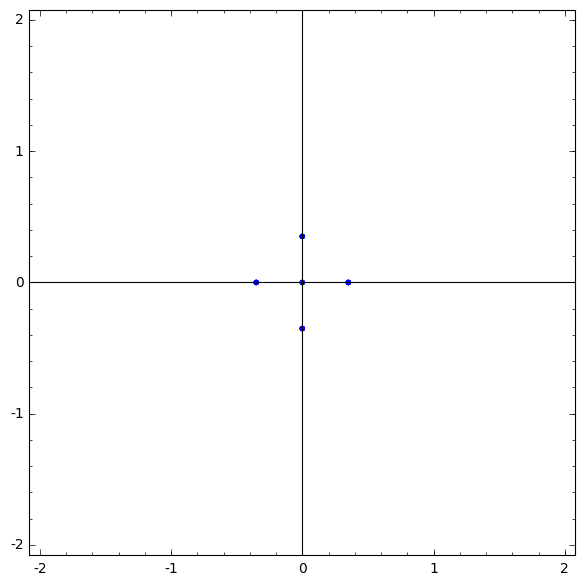}
  \includegraphics[width=0.3\textwidth]{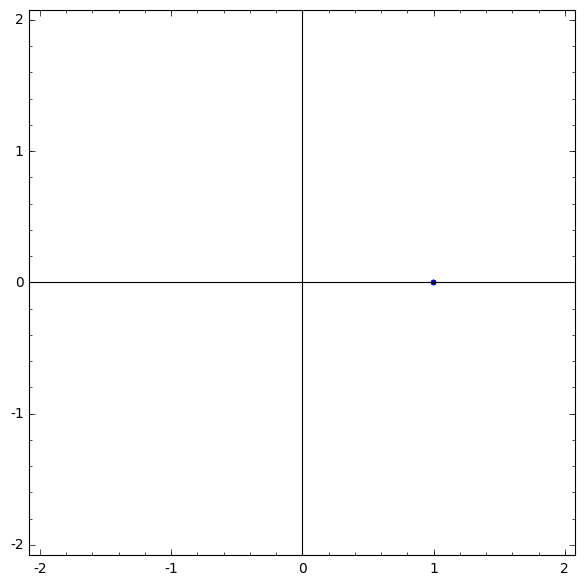}
  \includegraphics[width=0.3\textwidth]{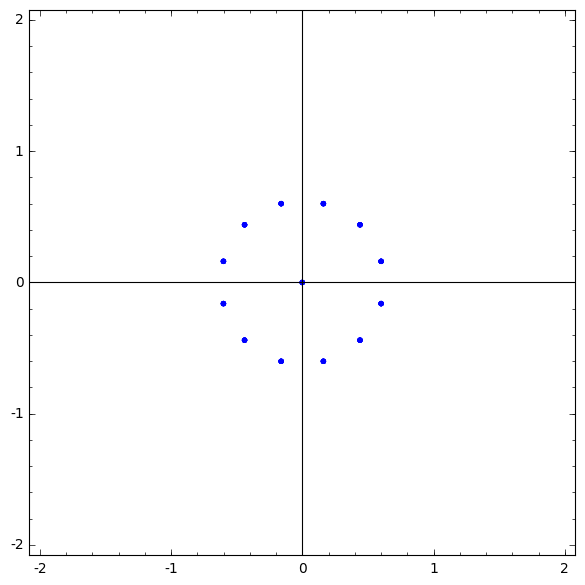}
  \includegraphics[width=0.3\textwidth]{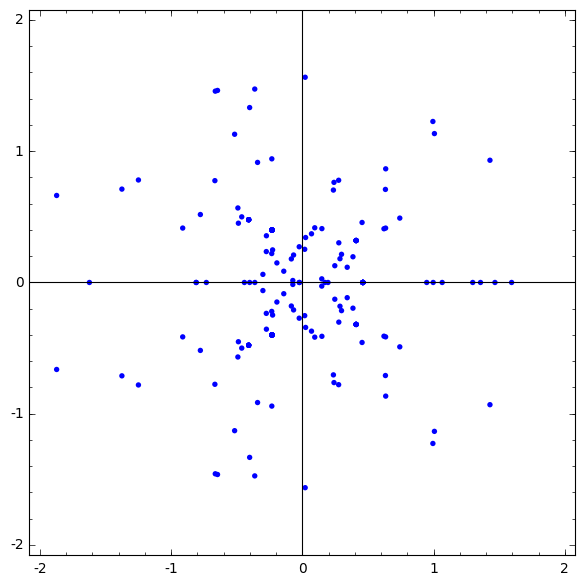}
  \includegraphics[width=0.3\textwidth]{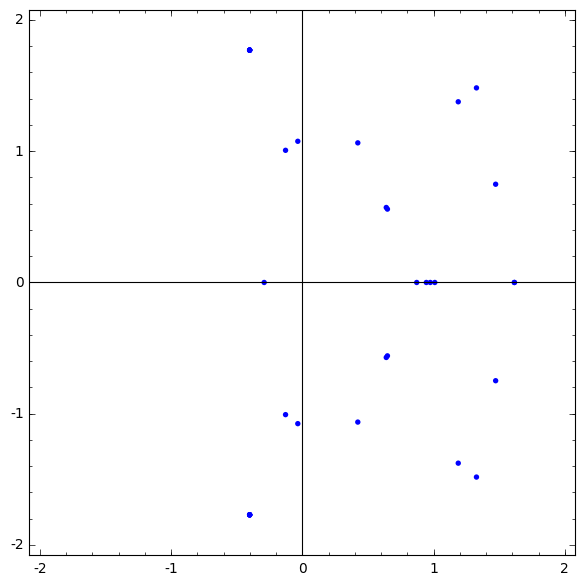}
  \includegraphics[width=0.3\textwidth]{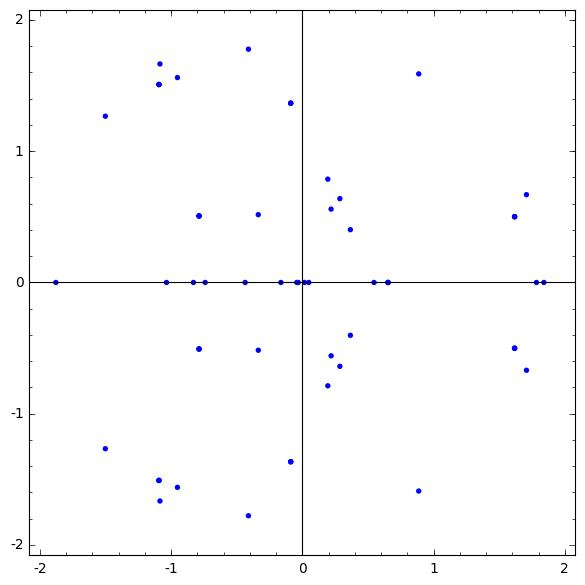}
  \caption{Unfavorable quartic: singularities of the ODEs for a monomial path}
  \label{fig:broken}
\end{figure}

\subsection{Higher genus curves}\label{sec:higher_genus}

\subsection*{A curve of genus 6}
We generated a random sparse degree 5 curve with small coefficients, getting 
\[
  -10x^5 + 3xy^3z - 2xz^4 - 2y^4z.
\]
A straight deformation from a Fermat curves appears hopeless. After some experimentation we found that the following path of deformation works well:
\[
  \begin{aligned}
  g_0 &= 10x^5 + y^5 + z^5              ,&   g_3 &= 10x^5 - 2xz^4 - 2y^4z + z^5 ,\\
  g_1 &= 10x^5 + y^5 - 2y^4z + z^5      ,&    g_4 &= 10x^5 - 2xz^4 - 2y^4z ,\\
  g_2 &= 10x^5 - 2y^4z + z^5            ,&    g_5 &= 10x^5 + 3xy^3z - 2xz^4 - 2y^4z.
  \end{aligned}
\]
The Picard--Fuchs equations take 14 seconds to find and 73 seconds to integrate to 30 digits of accuracy. Here {\tt algcurves} takes 64 seconds to compute the period matrix for 30 digits of accuracy.

\subsection*{A curve of genus 105}
As the periods of Fermat curves are readily available, the periods of any curve obtained from a ``small'' deformation of a Fermat type curve will be relatively easy to compute. Here, a small deformation means that the defining equation should differ only by a monomial or two from being a sum of powers, $ax^d+by^d+cz^d$. To illustrate this principle, let us compute the periods of the following genus 105 curve:
\[
  f_1=x^{16} + y^{16} + z^{16} + x^5y^5z^6.
\]
Choosing $f_0=x^{16} + y^{16} + z^{16}$ we compute the Picard--Fuchs equations of the family $(1-t)f_0+tf_1$ in 528 seconds. The orders of the 105 ODEs are either 12 or 14, while their degrees are less than 30. It takes another 317 seconds to integrate these equations with a final precision $10^{-25}$. In this case, {\tt algcurves} and {\tt abelfunctions} crash after a couple of minutes. With the newest patches, {\tt RiemannSurface} computes these integrals in just under 2 hours.

Needless to say, the $105\times 210$ period matrix is too big to display in print. However, since the total computation time is only 15 minutes, the interested reader can easily reproduce the results.

\subsection*{A curve of genus 1711}
Now consider the genus 1711 curve $f_1:= x^{60}+y^{60}+z^{60}+x^{20}y^{20}z^{20}$ with $f_0 =x^{60}+y^{60}+z^{60}$. Performed in isolation, it takes 8.5 seconds to compute all of 1711 Picard--Fuchs equations and 50 seconds to integrate them to 30 digits of precision. However, the period matrix has about six million entries which makes the storage and manipulation of this matrix a bottle neck. With our current implementation it takes over 2 hours for the computation to end.

\subsection{Periods of surfaces}\label{sec:surfaces}

We will now compute the classical periods for a few surfaces. In a future paper we will compute the period vectors for many more surfaces and determine their Picard rank and classify their transcendental lattices. 
Here, we will simply discuss issues regarding performance. 

\subsection*{A simple quartic K3} We take a random sparse quartic $f_1 = -3x^4 + 9xw^3 - 8y^3z - 4z^4 + w^4$ which is simple for us as it close to the Fermat quartic $f_0=-3x^4+y^4-4z^4+w^4$. Indeed, even the straight path from $f_0$ to $f_1$ gives a Picard--Fuchs equation of order 4 and degree 36, which takes 0.160 seconds to find and takes 5.985 seconds to integrate. The final result is precise to 20 decimal places. After truncating to 4 decimal places and denoting $\sqrt{-1}$ by $i$ we have:
\begingroup
\fontsize{7pt}{\baselineskip}\selectfont
\begin{multline*}
\big[-0.0128 + 0.4610i, -0.4056 + 0.2194i, -0.4056 + 0.2194i, -0.4056 - 0.2194i, 0.0128 - 0.4610i,0.0128 - 0.4610i,\\
  -0.0128 - 0.4610i, 0.0128 - 0.4610i, -0.0128 - 0.4610i, 0.0128 - 0.4610i,0.4056 - 0.2194i,  0.4056 + 0.2194i,\\
   0.4056 - 0.2194i,  0.4056 + 0.2194i,  0.4056 - 0.2194i,0.4056 + 0.2194i, 0.4056 - 0.2194i,-0.0128 + 0.4610i, \\
  0.0128 + 0.4610i,-0.0128 + 0.4610i,0.0128 + 0.4610i\big].
\end{multline*}
\endgroup

\subsection*{A fortunate monomial path} 
This example illustrates how drastic it can be to find a good monomial path. Take the quartic:
\[
  f_1=-x^4+2xy^3+2xw^3+10z^3w+3w^4,
\]
and let $f_0=-x^4+y^4+z^4+3w^4$. Then the computation of the Picard--Fuchs equation for the \emph{straight path} from $f_0$ to $f_1$ could not be found after 53 hours of computation.

On the other hand, consider the following monomial path:
\[
  \begin{aligned}
    g_0 &= -x^4+y^4+z^4+3w^4        ,&   g_3 &= -x^4+2xy^3+z^4+10z^3w+3w^4        ,\\
    g_1 &= -x^4+2xy^3+y^4+z^4+3w^4  ,&   g_4 &= -x^4+2xy^3+10z^3w+3w^4            ,\\
    g_2 &= -x^4+2xy^3+z^4+3w^4      ,&   g_5 &= -x^4+2xy^3+2xw^3+10z^3w+3w^4.
  \end{aligned}
\]
It takes 0.670 seconds to determine all of the $1+ 4\times 21 = 85$ ODEs we need. It takes 17 seconds to integrate these ODEs with precision $10^{-20}$, or 59 seconds with precision $10^{-100}$. The order-degree pairs of this system are as follows:
\begingroup
\fontsize{8pt}{\baselineskip}\selectfont
\[\arraycolsep=0.5pt
  \begin{array}{lccccccccccccccccccccc}
    G_1 \colon & (2,5) & (3,6) & (2,5) & (2,5) & (3,6) & (3,6) & (3,6) & (3,5) & (3,5) & (3,5) & (3,5) & (3,5) & (2,5) & (2,5) & (3,4) & (3,4) & (3,4) & (2,4) & (2,4) & (2,5) & (3,6)\\
    G_2 \colon & (2,3) & (3,5) & (3,4) & (3,4) & (2,4) & (2,4) & (3,3) & (3,3) & (3,3) & (3,3) & (3,3) & (3,3) & (2,3) & (2,3) & (3,3) & (3,3) & (3,3) & (2,4) & (2,4) & (2,3) & (3,5)\\
    G_3 \colon & (2,5) & (3,6) & (3,5) & (3,6) & (2,4) & (2,5) & (3,4) & (3,5) & (3,6) & (3,4) & (3,5) & (3,6) & (3,5) & (2,5) & (3,4) & (3,5) & (3,6) & (3,5) & (2,5) & (3,6) & (3,6)\\
    G_4 \colon & (2,3) & (3,4) & (3,5) & (3,4) & (2,4) & (2,3) & (3,3) & (3,3) & (3,3) & (3,3) & (3,3) & (3,3) & (3,4) & (2,4) & (3,3) & (3,3) & (3,3) & (3,4) & (2,4) & (3,5) & (3,5)\\
    G_5 \colon & (4,4)&&&&&&&&&&&&&&&&&&&& 
\end{array}
\]
\endgroup

\subsection*{A harder quartic} 
The quartic surface defined by $f_1=9x^3z + 4y^3w + 3z^4 + 8z^2w^2 + 2w^4$ appears to be harder: it takes 3.4 seconds for the Picard--Fuchs equations to be computed and 101 seconds to integrate them to 33 digits of accuracy. 

\subsection*{A quintic surface} 
We took a randomly generated quintic surface $f_1=6x^4z + 3y^4w - 9y^3z^2 + 3z^5 - w^5$. It took 74 seconds to determine the Picard--Fuchs equations and 7.5 minutes to integrate these equations to determine the period matrix to 24 digits of accuracy.

\subsection{A cubic threefold}  \label{sec:threefold}
Degree three hypersurfaces are typically favorable for our algorithm, since the Griffiths--Dwork reductions are easy to perform. Indeed, the periods of a sparse cubic are almost always quickly computed even if the defining polynomial of the cubic is very different from a Fermat type polynomial.

However, for our demonstration we choose a relatively simple cubic threefold so that the computation of its Picard--Fuchs equations are fast and we can concentrate on the time cost of increasing precision. Even though the Picard--Fuchs equations are found quickly, they are not particularly easy to integrate because the singular fibers are not favorably positioned. 

Letting $f_1 = -8x^2w-8y^3+z^3-9zs^2+w^3$ we find the following sequence of hypersurfaces where only one monomial is changed at each step:
\[
  \begin{aligned}
  g_0&=x^3-8y^3+z^3+w^3+s^3           ,&      g_3&=x^3-8x^2w-8y^3+z^3-9zs^2+w^3 ,\\
  g_1&=x^3-8y^3+z^3-9zs^2+w^3+s^3     ,&      g_4&=-8x^2w-8y^3+z^3-9zs^2+w^3. \\
  g_2&=x^3-8y^3+z^3-9zs^2+w^3         ,&    
  \end{aligned}
\]
The Picard--Fuchs equations are computed in 0.230 seconds and it takes 6 seconds to integrate the system with 20 digits of precision, or 620 seconds with 1000 digits of precision. The time versus precision graph is plotted in Figure~\ref{fig:extreme_precision}. Of course, the integrator {\tt oa-analytic}~\cite{mezzarobba-oa} does all the heavy lifting here.

\begin{figure}
  \includegraphics[width=0.7\textwidth]{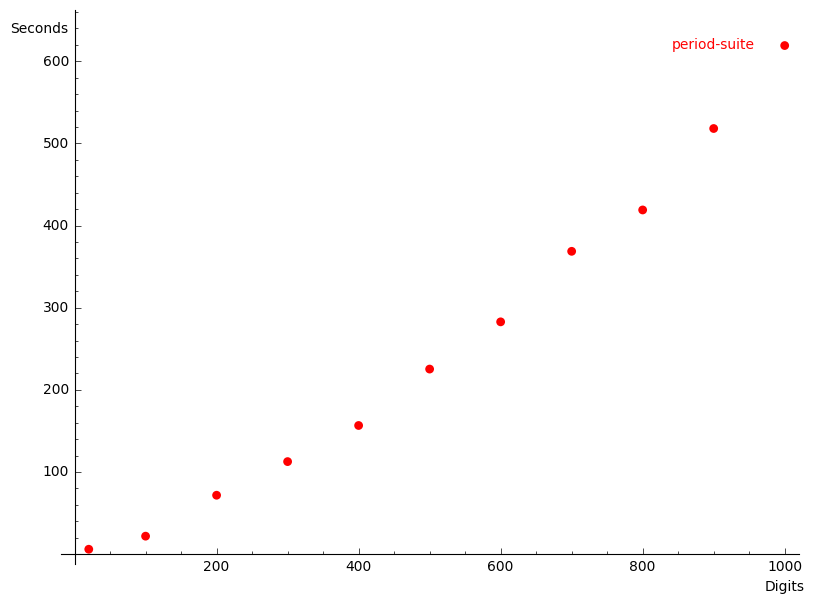}
  \caption{Time vs.\ precision plot for a cubic threefold}
  \label{fig:extreme_precision}
\end{figure}

\printbibliography

\end{document}

%% file: preamble.tex
\usepackage[utf8]{inputenc}
\usepackage[T1]{fontenc} 
\usepackage{microtype}
\usepackage[american]{babel}
\usepackage{csquotes}
\usepackage{pdfpages}
\usepackage{emptypage}
\usepackage{geometry}

\usepackage{listings}
\usepackage{color}
\usepackage{xcolor}
\usepackage{listings}

\usepackage{caption}
\DeclareCaptionFont{white}{\color{white}}
\DeclareCaptionFormat{listing}{\colorbox{gray}{\parbox{\textwidth}{#1#2#3}}}
\usepackage[normalem]{ulem}

\usepackage{amsmath}
\usepackage{mathtools}
\usepackage{amsthm}
\usepackage{amssymb}
\usepackage{enumitem}

\usepackage{hyperref}
\usepackage[style=alphabetic,citestyle=alphabetic,backend=bibtex]{biblatex}

\usepackage{tikz-cd}
\usepackage{chngcntr}

\usepackage{amsfonts}
\usepackage{yfonts}
\usepackage{mathrsfs}
\usepackage{bbm}
\usepackage[%
  cal=euler,
  scr=rsfso%
]{mathalfa} 

\usetikzlibrary{decorations.pathmorphing}

\DeclareMathOperator{\aut}{Aut}

\DeclareMathOperator{\res}{res}

\newcommand{\dd}{\mathrm{d}}

\DeclareMathOperator{\dR}{dR}

\DeclareMathOperator{\pcl}{pcl}
\DeclareMathOperator{\GD}{GD}

\usepackage{array}   
\newcolumntype{M}{>{$}c<{$}} 
\newcolumntype{L}{>{$}l<{$}}
\usepackage{arydshln}
\usepackage{multirow}

\DeclareMathOperator{\Hom}{\mathscr{H}\kern -2pt om}
\DeclareMathOperator{\Ext}{\mathscr{E}\kern -1.5pt xt}

\newtheorem{theorem}{Theorem}[section]
\numberwithin{theorem}{section} 

\newcommand{\nameofthm}{}

\newtheorem{genericThm}[theorem]{\nameofthm}

{%
  \begin{proof}$ $\newline
  \indent ($\implies$) 
}
{\end{proof}}

\newtheorem*{theorem*}{Theorem}

\newtheorem{lemma}[theorem]{Lemma}
\newtheorem{corollary}[theorem]{Corollary}
\newtheorem{proposition}[theorem]{Proposition}

\theoremstyle{definition}

\newtheorem*{definition*}{Definition}
\newtheorem{definition}[theorem]{Definition}
\newtheorem{notation}[theorem]{Notation}
\newtheorem{remark}[theorem]{Remark}

\newtheorem*{notation*}{Notation}

\newcommand{\nameofenv}{}
\newtheorem*{generic}{\nameofenv}

\newcommand{\nameOfNumEnv}{}
\newtheorem{genericNumbered}[theorem]{\nameOfNumEnv}
\newenvironment{namedNum}[1]
{%
  \renewcommand{\nameOfNumEnv}{#1}%
  \begin{genericNumbered}%
}
{%
  \end{genericNumbered}%
}

\makeatletter
\def\imod#1{\allowbreak\mkern10mu({\operator@font mod}\,\,#1)}
\makeatother

\newcommand{\del}{\partial}

\newcommand{\p}{\Pp^1}
\newcommand{\pp}{\Pp^2}
\newcommand{\ppp}{\Pp^3}

\newcommand{\inv}{^{-1}}
\newcommand{\set}[1]{\{ #1 \}}
\DeclarePairedDelimiter{\ceil}{\lceil}{\rceil}

\newcommand{\tos}{\twoheadrightarrow}
\newcommand{\toi}{\hookrightarrow}
\newcommand{\isoto}{\overset{\sim}{\to}}

\newcommand{\Cc}{\mathbbm{C}}

\newcommand{\Kk}{\mathbbm{K}}

\newcommand{\Qq}{\mathbbm{Q}}
\newcommand{\Pp}{\mathbbm{P}}
\newcommand{\Rr}{\mathbbm{R}}

\newcommand{\Zz}{\mathbbm{Z}}

\renewcommand{\H}{\mathrm{H}}
\newcommand{\PH}{\mathrm{PH}}

\newcommand{\xc}{\mathfrak{c}}

\newcommand{\xj}{\mathfrak{j}}

\newcommand{\cd}{\mathcal{D}}

\newcommand{\cl}{\mathcal{L}}

\newcommand{\co}{\mathcal{O}}

\newcommand{\cp}{\mathcal{P}}

\newcommand{\cx}{\mathcal{X}}

